\documentclass[10pt]{article}
\usepackage[utf8]{inputenc}
\usepackage[english]{babel}
\usepackage{amsmath}
\usepackage{amssymb}
\usepackage{amsfonts}
\usepackage{amsthm}
\usepackage{enumerate}
\usepackage{tikz-cd}
\usepackage{subfiles}
\usepackage[scr]{rsfso}
\usepackage{extarrows}
\usepackage{hyperref}
\usepackage[titletoc]{appendix}
\definecolor{allrefcolors}{rgb}{0,0.5,0.4}
\hypersetup{
    colorlinks=true,
    linkcolor=allrefcolors,
    filecolor=allrefcolors,      
    urlcolor=allrefcolors,
    citecolor=allrefcolors
}
\title{Rel--$C^\infty$ Structures on Gromov-Witten Moduli Spaces}
\author{Mohan Swaminathan}
\newtheorem{theorem}{Theorem}[section]
\newtheorem{lemma}[theorem]{Lemma}
\newtheorem{corollary}[theorem]{Corollary}
\newtheorem{proposition}[theorem]{Proposition}

\theoremstyle{definition}
\newtheorem{definition}[theorem]{Definition}
\theoremstyle{remark}
\newtheorem{remark}[theorem]{Remark}

\numberwithin{equation}{subsection}

\makeatletter
\@addtoreset{equation}{section}
\@addtoreset{theorem}{section}
\makeatother

\newcommand{\Mbar}{\overline{\mathcal M}}

\newcommand{\delbar}{\bar\partial}

\newcommand{\Hom}{\normalfont\text{Hom}}

\newcommand{\Aut}{\normalfont\text{Aut}}

\newcommand{\bC}{\mathbb{C}}
\newcommand{\bD}{\mathbb{D}}

\newcommand{\bR}{\mathbb{R}}

\newcommand{\bZ}{\mathbb{Z}}
\newcommand{\cA}{\mathcal{A}}
\newcommand{\cB}{\mathcal{B}}
\newcommand{\cC}{\mathcal{C}}

\newcommand{\cL}{\mathcal{L}}
\newcommand{\cM}{\mathcal{M}}
\newcommand{\cN}{\mathcal{N}}

\newcommand{\cS}{\mathcal{S}}

\newcommand{\cU}{\mathcal{U}}
\newcommand{\cV}{\mathcal{V}}

\newcommand{\cZ}{\mathcal{Z}}

\newcommand{\fC}{\mathfrak{C}}

\newcommand{\fF}{\mathfrak{F}}

\newcommand{\fM}{\mathfrak{M}}

\newcommand{\fS}{\mathfrak{S}}

\newcommand{\fU}{\mathfrak{U}}

\newcommand{\fX}{\mathfrak{X}}
\newcommand{\fY}{\mathfrak{Y}}

\newcommand{\fa}{\mathfrak{a}}

\newcommand{\fo}{\mathfrak{o}}

\newcommand{\ft}{\mathfrak{t}}
\newcommand{\fu}{\mathfrak{u}}

\newcommand{\fz}{\mathfrak{z}}
\newcommand{\PShv}{\normalfont\text{PShv}}

\begin{document}
	\maketitle
	\begin{abstract}
	    We show that moduli spaces of transversely cut-out (perturbed) pseudo-holomorphic curves in an almost complex manifold carry canonical relative smooth structures (``relative to the moduli space of domain curves"). The main point is that these structures can be characterized by a universal property. The tools required are ordinary gluing analysis combined with some fundamental results from the polyfold theory of Hofer--Wysocki--Zehnder.
	\end{abstract}
	\tableofcontents
	\section{Introduction}
	\subsection{Main Result}
	Constructing smooth structures on transversely cut-out moduli spaces of holomorphic curves is an important problem (see \cite{Abou}, \cite{Ek-Sm} and \cite{Floer-Homotopy} for applications) with the main difficulty lying in constructing the smooth structure near the nodal curves. More precisely, in the presence of transversality, arguments using the implicit function theorem only show that each individual stratum of the moduli space is smooth while standard gluing analysis shows how these strata fit together \emph{topologically}. The polyfold theory of Hofer--Wysocki--Zehnder provides one possible approach to constructing smooth structures on moduli spaces of holomorphic curves (in addition to providing a framework to address the case where transversality is not present). However, it is not clear that the smooth structure we get this way is independent of various choices made in the construction. Another approach to construct smooth structures is described in \cite{FOOO-gluing}. Again, it is not clear that this produces the same smooth structures near the nodal curves as the polyfold approach. In this article, we construct \emph{canonical} relative smooth structures (a weaker notion than smooth structures -- formally defined in Definition \ref{rel-Cr-mflds})	on the moduli spaces of	regular (i.e. transversely cut-out) solutions to the (perturbed) Cauchy--Riemann equation, relative to the space of domain curves.\\\\
	There is one specific situation in which defining canonical smooth structures on moduli spaces is quite easy -- namely the case in	which the domain curve is fixed. If $C$ is a compact nodal Riemann surface and $(X,J)$ is an almost complex manifold, we may consider the space $\fM_J(C,X)$ of $J$-holomorphic maps $u:C\to X$. We say $u$ is \textbf{cut-out transversely} (or \textbf{regular}) iff the linearized
	Cauchy-Riemann operator 
	\begin{align}
		D(\delbar_J)_u:\Omega^0(C,u^*T_X)\to\Omega^{0,1}_J(\tilde C,u^*T_X)
	\end{align}
	is surjective. (Here, $\tilde C$ is the normalization of $C$, i.e., the curve got by separating the two smooth branches meeting transversely at each node of $C$.) Taking an integer $k\ge 1$ and a real number $1<p<\infty$ with $kp>2$, we can realize $\fM_J(C,X)$ as $\delbar_J^{-1}(0)\subset W^{k,p}(C,X)$. The subset $\fM^\text{reg}_J(C,X)$ consisting of	regular $u\in\fM_J(C,X)$ is open and naturally carries the structure of a smooth manifold (and this structure can be shown	to be independent of the initial choice of $k,p$ -- this also follows from Theorem \ref{rep-intro} below). \\\\
	Next, given a proper flat analytic family $\pi:\cC\to\cS$ of possibly nodal Riemann surfaces (with the fibre over $s\in\cS$ denoted by $C_s$), parametrized by a complex manifold $\cS$, we can consider the parametrized moduli space $\fM_J^\text{reg}(\pi,X)$ consisting of pairs $(s\in\cS,f:C_s\to X)$, where $f$ is $J$-holomorphic and regular. Following the discussion from the previous paragraph, we could realize this moduli space as an open subset of $\delbar_J^{-1}(0)$ inside the space $W^{k,p}(\pi,X) = \bigsqcup_{s\in\cS}W^{k,p}(C_s,X)$ consisting of pairs $(s\in\cS,f:C_s\to X)$ with $f$ of class $W^{k,p}$. A key technical difficulty here is that $W^{k,p}(\pi,X)$ does not carry any natural Banach manifold structure. There are two main reasons for this. Firstly, the topology of $C_s$ itself may change as we vary $s\in\cS$ (due to the formation of nodes). Secondly, even near a regular value $s\in\cS$ of $\pi$, the obvious attempt to define a smooth structure on $W^{k,p}(\pi,X)$ -- namely by identifying it near $s$ with $W^{k,p}(C_s,X)\times\cU$ after choosing a local $C^\infty$ bundle trivialization of $\pi:\cC\to\cS$ over a neighborhood $\cU$ of $s$ -- fails. This is because the reparametrization action
	\begin{align}
	    \text{Diff}(C_s)\times W^{k,p}(C_s,X)&\to W^{k,p}(C_s,X)\\
	    (\varphi,f)&\mapsto f\circ\varphi
	\end{align}
	is not differentiable in the variable $\varphi$ (see e.g. \cite[Example 2.4]{FFGW-polyfold-survey}) and thus, different local $C^\infty$ bundle trivializations lead to charts for $W^{k,p}(\pi,X)$ which are not smoothly compatible. On the other hand, this discussion shows that it is reasonable to expect that $\fM^\text{reg}_J(\pi,X)$ carries a smooth structure ``relative to $\cS$", i.e., the structure of a space over $\cS$ whose fibres are all smooth manifolds with the smooth structures on the fibres ``varying continuously" with $s\in\cS$. This is the statement we prove in this article, formulated more precisely below.
	\begin{theorem}[see Theorem \ref{rep}]\label{rep-intro}
		Let $\pi:\cC\to\cS$ be a proper flat analytic family of prestable\footnote{A prestable curve $(C,x_1,\ldots,x_m)$ is a connected projective algebraic curve over $\bC$		with at worst ordinary nodal singularities and equipped with a finite (possibly empty) collection of distinct smooth marked points $x_1,\ldots,x_m$. A prestable curve is called stable if its automorphism group $\Aut(C,x_1,\ldots,x_m)$, consisting of biholomorphisms $\varphi:C\to C$ with $\varphi(x_i)=x_i$, is finite.} curves over a complex manifold $\cS$ and let $(X,J)$ be an almost complex manifold. Then, the space $\normalfont\fM_J^\text{reg}(\pi,X)$ of regular $J$-holomorphic maps from fibres of $\pi$ carries the		structure of a relative smooth manifold over $\cS$. This relative smooth structure can be characterized by a universal property and is compatible under pulling back the family of curves along holomorphic maps $\cS'\to\cS$.
	\end{theorem}
	\begin{remark}
	    In fact, Theorem \ref{rep-intro} as stated above implies, purely formally, the corresponding statement for $\cS$ a smooth Deligne--Mumford (or Artin) stack. Namely, given the smooth Deligne--Mumford (or Artin) stack $\cS$, choose an \'etale (or smooth) atlas $U\to\cS$, apply Theorem \ref{rep-intro} to $U$. Note that the two pullback structures on $U\times_{\cS}U$ agree by the naturality statement in Theorem \ref{rep-intro}, hence this descends to the stack $\cS$ -- also note that the same argument shows independence of the choice of the atlas $U$. Of course, there is a `universal' choice of $\cS$, namely the smooth Artin stack $\fM^\text{pre}$ of all prestable curves. Everything is obtained via pullback from it and so, Theorem \ref{rep-intro} could be equivalently stated by saying that we have a relative smooth structure over this universal base. In other words, if $\fM_J(X)\to\fM^\text{pre}$ is the topological stack of $J$-holomorphic maps from prestable curves and $\fM_J^\text{reg}(X)\subset\fM_J(X)$ is the open substack of regular maps, then Theorem \ref{rep-intro} asserts the existence of a relative smooth structure on $\fM_J^\text{reg}(X)\to\fM^\text{pre}$.
	\end{remark}
	\begin{remark}
		The actual result we prove (namely, Theorem \ref{rep}) is somewhat more general as it allows for inhomogeneous perturbation terms in the Cauchy-Riemann equation. In this case also, the relative smooth structure we construct is compatible under pullbacks of families of curves and inhomogeneous terms.
	\end{remark}
	\begin{remark}
		Since the relative smooth structure we construct can be characterized by a simple universal property (see		Definition \ref{def-mod-functor}), this structure is the ``correct" one. Thus, even though we prove the existence of this structure using techniques from polyfold theory -- which involves making choices of a gluing profile and a specific functional analytic setup involving weighted Sobolev spaces -- the result is still \emph{independent of all choices}. We would like to emphasize that this is quite different from the usual way in which moduli spaces are constructed in symplectic geometry -- the construction \emph{a priori} depends strongly on the specific framework used (polyfolds, Kuranishi structures etc) and thus, a lot of effort is required to prove that the local charts thus constructed are smoothly compatible. We, on the other hand, start with an easy-to-state universal property which characterizes the relative smooth structure uniquely \emph{provided that it exists}. The burden is thus shifted to a proof that such a structure exists, and this is tackled using techniques familiar from holomorphic curve theory. In this sense, our ``functor of points" approach to moduli spaces closely resembles the one adopted in algebraic geometry.
	\end{remark}
	\begin{remark}[Non-transverse Case] In general, when the moduli space of holomorphic curves is not transversely cut-out, the question of constructing smooth structures on it doesn't make sense as stated. However, we may instead like to construct smooth Kuranishi structures or atlases on these moduli spaces (in the sense of \cite{FOOO1}, \cite{FOOO2} or \cite{MW}). Since our main result (Theorem \ref{rep}) deals with the perturbed Cauchy--Riemann equation \eqref{inhom-pde}, it can be used (in conjunction with \cite[Lemma 9.2.9]{Pardon-VFC}) to construct relative smooth Kuranishi structures/atlases (or implicit atlases, in the terminology of \cite{Pardon-VFC}) on non-regular moduli spaces as well.
	\end{remark}
	\subsection{Guide to the Article}
	In \textsection\ref{rel}, we define the category of relative smooth manifolds and morphisms between them. We also include a brief discussion of representable (set-valued) functors defined on this category for the convenience of the reader.
	\\\\
	In \textsection\ref{mod-problem}, we introduce the moduli functors of interest to us. This gives a categorical way of stating the universal property we wish our construction of relative smooth structures to satisfy. We then discuss the underlying topology of the moduli problem. 
	\\\\
	In \textsection\ref{sc-ift}, which is the heart of the article, we prove the main existence	result (Theorem \ref{rep}). The proof makes use of techniques from classical gluing analysis for holomorphic curves (e.g. as described in \cite[Appendix B]{Pardon-VFC}) and some fundamental results of polyfold theory (specifically, from \cite{HWZ-GWbook} and \cite{HWZ-ImpFuncThms}). In this section, we assume the reader has a working familiarity with the basics of polyfold theory	(e.g. at the level of \cite[\textsection4--\textsection6]{FFGW-polyfold-survey}). Once the proof of existence is completed, we do not need to revisit this construction as only its universal property is relevant for the rest of the article.
	\\\\
	In \textsection\ref{mod-with-ob}, we extend the result of the previous section to treat another common setup that occurs in holomorphic curve theory -- namely the construction	of certain types of obstruction bundles on moduli spaces. The main result of this section (Theorem \ref{obs-thm}) constructs natural rel--$C^\infty$ structures on certain obstruction bundles on moduli spaces of holomorphic curves.
	\\\\
	In \textsection\ref{stable-mod}, we show that stable maps are open in the moduli spaces we have constructed in earlier sections. As a byproduct of the technique used to prove this,	we also describe local \'etale charts for the stack of stable maps (including when it is not cut-out transversely).
	\\\\
	Finally, Appendix \ref{wk-appx} collects some technical results used in \textsection\ref{sc-ift}.
	\subsection{Acknowledgements}
	I thank my advisor John Pardon for many useful conversations and comments on earlier versions of this article. I thank Michael Jemison for answering many of my basic questions on polyfold theory and explaining the gluing theorem for holomorphic curves to me in this context. I thank Helmut Hofer for useful discussions and encouragement. I am also grateful for the useful conversations I had with Shaoyun Bai, Nate Bottman, Yash Deshmukh, Thomas Massoni, Rohil Prasad and Zhengyi Zhou. 
	\section{Relative $C^r$ Manifolds}\label{rel}
	\subsection{Basic Definitions}
	\noindent We start by defining a class of geometric spaces which should be viewed as a ``relative version" of the class of differentiable manifolds (of class $C^r$ for some integer $r\ge 0$ or $r=\infty$).
	\begin{definition}[Relative Spaces and Morphisms]
		\begin{enumerate}[(1)]
			\item Let $S$ be a topological space. A \textbf{space relative to $S$} (also referred to as an			\textbf{$S$-space}) is the data of a topological space $X$ equipped with a continuous map $f:X\to S$,			called the \textbf{structure morphism}. 			When $f$ is clear from context or unimportant, we will abbreviate $(X,f)$ to $X/S$ or simply $X$. An $S$-space $(X,f)$ is called \textbf{separated} if the diagonal map
			\begin{align}
				X\hookrightarrow X\times_SX
			\end{align}
			is a closed embedding.
			\item A \textbf{map of $S$-spaces} (also referred to as an \textbf{$S$-morphism}) from $(X,f)$ to $(X',f')$	is a continuous map $g:X\to X'$ such that the following diagram commutes.
			\begin{center}
			\begin{tikzcd}
				X \arrow[rr,"g"] \arrow[dr, "f"]& &X' \arrow[dl, "f'"]\\
				& S &
			\end{tikzcd}
			\end{center}
			\item Given another topological space $S'$, a continuous map $\varphi:S'\to S$ and an $S$-space $(X,f)$, we define the \textbf{pullback} $\varphi^*(X,f)$ to be the $S'$-space $(X',f')$ given by the following fibre product diagram.
			\begin{center}
			\begin{tikzcd}
				X' \arrow[r, "\tilde\varphi"] \arrow[d,"f'"] & X \arrow[d, "f"] \\
				S' \arrow[r,"\varphi"] & S
			\end{tikzcd}
			\end{center}
			\item Given an $S$-space $(X,f)$ and an $S'$-space $(X',f')$, a \textbf{relative morphism} from $(X,f)$ to $(X',f')$ is a pair $(\tilde\varphi:X'\to X,\varphi:S'\to S)$ of continuous maps such that the square above is commutative (but not necessarily a fibre square).
			\item Given a topological space $Y$, the \textbf{trivial $S$-space with fibre $Y$} is defined to be $Y_S:=(Y\times S,\pi_S)$, where $\pi_S:Y\times S\to S$ is the second coordinate projection.
		\end{enumerate}
	\end{definition}
	\begin{definition}[Rel--$C^r$ Maps]
		Consider a relative morphism $\Phi=(\tilde\varphi,\varphi):\cU\to V'_{S'}$, where $\cU$ is an open $S$-subspace of the $V_S$, with $V,V'$ being finite dimensional $\bR$-vector spaces. We then define the notion of $\Phi$ being of class \textbf{rel--$C^r$} for $r\ge 0$ by induction on $r$. For $r=0$, we require no		additional conditions on $\Phi$ (i.e., $\Phi$ is always of class rel--$C^0$). For $r\ge 1$, we define $\Phi$ to be of class rel--$C^r$ if and only if the following conditions hold.
		\begin{itemize}
			\item Writing $\tilde\varphi(x,s) = (F(x,s),\varphi(s))$, the map $x\mapsto F(x,s)$ is differentiable (as a map between open subsets of vector spaces) for each point $s\in S$.
			\item The induced tangent map $T\Phi=(T\tilde\varphi,\varphi):(T_U)_{S}|_\cU\to (T_{V'})_{S'}$ given by
			\begin{align}
				T\tilde\varphi:(x,v,s)\mapsto(F(x,s),\nabla_v F(x,s),\varphi(s))
			\end{align} 
			is of class rel--$C^{r-1}$ (where $T_M$ denotes the tangent bundle of a manifold $M$).
		\end{itemize}
		We say $\Phi$ is \textbf{relatively smooth} (or of class \textbf{rel--$C^\infty$}) if and only if $\Phi$ is of class	rel--$C^r$ for all integers $r\ge 1$.
	\end{definition}
	\begin{lemma}\label{rel-chain}
		For integers $r\ge 0$ (as well as $r=\infty$), a composition of rel--$C^r$ maps is again rel--$C^r$.
	\end{lemma}
	\begin{proof}
		Obvious by the ordinary chain rule and induction on $r$.
	\end{proof}
	\begin{definition}[Rel--$C^r$ Compatibility of Charts]
		Let $S$ be a topological space, $n\ge 0$ be an integer and $(X,f)$ be an $S$-space. An \textbf{$S$-chart (of relative dimension $n$) for $X$} is a pair $(U,\varphi)$ consisting of an open subset $U\subset X$ along with an $S$-morphism $\varphi:U\to\bR^n_S$ which is an homeomorphism onto an open subset of $\bR^n_S$. Now, given two $S$-charts $(U,\varphi)$ and $(V,\psi)$ of relative dimension $n$ and $0\le r\le\infty$, the two charts are said to be \textbf{rel--$C^r$ compatible} with each other if the both map
		\begin{align}
			\psi\circ\varphi^{-1}:\varphi(U\cap V)\to\psi(U\cap V)
		\end{align}
		and its inverse are of class rel--$C^r$. A \textbf{rel--$C^r$ atlas} for $X/S$ is a collection $\cA$ of charts $\{(U_\alpha,\varphi_\alpha)\}_{\alpha}$ such that $\{U_\alpha\}_\alpha$ is an open cover of $X$ and		any two charts in $\cA$ are rel--$C^r$ compatible. Elements of a rel--$C^r$ atlas are called rel--$C^r$ charts.
	\end{definition}
	\begin{definition}[Rel--$C^r$ Manifolds and Maps]\label{rel-Cr-mflds}
		Given a topological space $S$, an integer $n\ge 0$, $0\le r\le\infty$ and an $S$-space $X$, a \textbf{rel--$C^r$ manifold} structure $\cA$ on $X/S$ (of relative dimension $n\ge 0$) is a maximal rel--$C^r$ atlas on $X$ (all of whose $S$-charts are of relative dimension $n$). We say that the rel--$C^r$ manifold $X/S$ is \textbf{separated}		if the structure map $X\to S$ is separated. We also include the case where $X$ is the empty set (in this case, the relative dimension over $S$ is not well-defined). The notion of rel--$C^k$ maps between rel--$C^r$ manifolds for $0\le k\le r$ is defined in the obvious fashion using atlases.
	\end{definition}
	\begin{remark}\label{rel-over-pt}
	    We include ordinary $C^r$ manifolds $X$ into rel--$C^r$ manifolds by taking $S=\text{pt}$.
	\end{remark}
	\begin{definition}[Base Change]\label{rel-bc}
		Let $X/S$ be a rel--$C^r$ manifold and let $T\to S$ be a continuous map of topological spaces. Then, the space $X_T:=X\times_ST$ carries the natural structure of a rel--$C^r$ manifold over $T$. This is evident when $X/S = V_S$  for some vector space $V$ (or an open subset thereof), in which case $X_T/T$ is canonically identified with $V_T$ (respectively, an open subset thereof). For general $X/S$, we define the rel--$C^r$ structure on $X_T/T$ by gluing these local models. We call		$X_T/T$ the \textbf{base change of $X/S$ along $T\to S$}.
	\end{definition}
	\begin{definition}
		Let $0\le r\le \infty$. We define the \textbf{category of rel--$C^r$ manifolds}, denoted by $(C^r/\cdot)$, to be the category whose objects are rel--$C^r$ manifolds and whose morphisms are rel--$C^r$ maps between them. Given any		topological space $S$, the subcategory of $(C^r/\cdot)$ consisting of rel--$C^r$ manifolds which are $S$-spaces and rel--$C^r$ $S$-morphisms between them is called the \textbf{category of rel--$C^r$ $S$-manifolds} and denoted by $(C^r/S)$. Note that this is not a full subcategory.
	\end{definition}
	\subsection{Categorical Properties}
	We note here some basic properties of the category $(C^r/\cdot)$ which will be useful later. We begin by discussing some generalities on categories and then specialize to the case at hand.
	\subsubsection{Adjoints, Presheaves and Yoneda Functors}
	\label{abstract-nonsense}
	For a category $\cA$, we let $\PShv(\cA)$ denote the category of (set-valued) presheaves on $\cA$, i.e., the category of contravariant functors from $\cA$ to the category $(\text{Sets})$. Let $h:\cA\hookrightarrow\PShv(\cA)$ be the Yoneda embedding which maps an object $X$ of $\cA$ to the object $h_X:=\Hom_\cA(-,X)$ of $\PShv(\cA)$. The essential image of $h$ is the subcategory of \textbf{representable} presheaves.
	\begin{remark}\label{yoneda-lemma}
		We will adhere to the convention of identifying an object $X$ with the associated Yoneda presheaf $h_X$. In accordance with the Yoneda Lemma \cite[\href{https://stacks.math.columbia.edu/tag/001P}{Tag 001P}]{stacks-project}, for a presheaf $\fX$, we will identify elements of $\fX(X)$ with morphisms $X\to\fX$ of presheaves, under the natural bijection
		\begin{align}\label{yoneda-correspondence}
		    \text{Mor}_{\PShv(\cA)}(X,\fX)&\xrightarrow{\simeq}\fX(X)\\
		    \Phi&\mapsto\Phi_X(\text{id}_X).
		\end{align}
		In practice, when defining a presheaf $\fF$ on $\cA$, we will specify what ``morphisms from $A$ to $\fF$" are for every object $A$ in $\cA$ (in other words, we will specify the sets $\fF(A)$) and how to ``pull back a morphism $A\to\fF$ to a morphism $A'\to\fF$" for a given morphism $A'\to A$ in $\cA$ (in other words, the functorial set maps $\fF(A)\to\fF(A')$).
	\end{remark}
	\begin{definition}
	    Suppose $F_1,F_2:\cA\to\cB$ are two functors (where $\cA$ and $\cB$ are categories) and $\eta:F_1\Rightarrow F_2$ is a natural transformation. Assume that $F_1,F_2$ admit left adjoints $F_1',F_2'$, i.e., we have natural bijections
	    \begin{align}
	        \psi_i(A,B):\Hom_\cA(F_i'(B),A)\xrightarrow{\simeq}
	        \Hom_\cB(B,F_i(A))
	    \end{align}
	    for $i=1,2$ and all $A\in\cA$ and $B\in\cB$.
	    Let $\eta':F_2'\Rightarrow F_1'$ be the induced natural transformation. For $i=1,2$, these induce pullback functors $F_i'':\PShv(\cA)\to\PShv(\cB)$ and a natural transformation
	    \begin{align}
	        \eta'':F_1''\Rightarrow F_2''.
	    \end{align}
		In view of the next lemma, we say that $\eta'':F_1''\Rightarrow F_2''$ is the \textbf{extension of $\eta:F_1\Rightarrow F_2$ to presheaves}.
	\end{definition}
	\begin{lemma}\label{morphism-of-PShv}
	    In the situation above, $\eta:F_1\Rightarrow F_2$ and $\eta'':F_1''\Rightarrow F_2''$ are naturally compatible under the Yoneda embeddings $\cA\hookrightarrow\PShv(\cA)$ and $\cB\hookrightarrow\PShv(\cB)$. In particular, given any object $A\in\cA$, we have the natural isomorphisms
	    \begin{align}
	        \alpha_i(A):=\psi_i^{-1}(A,-):h_{F_i(A)}\Rightarrow h_A\circ F_i'=:F_i''(h_A)
	    \end{align}
	    for $i=1,2$ in $\PShv(\cB)$. Furthermore, these fit into the following commutative diagram.
		\begin{center}
			\begin{tikzcd}
				h_{F_1(A)} \arrow[r, "h_{\eta(A)}"] \arrow[d, "\alpha_1(A)"] & 
				h_{F_2(A)} \arrow[d, "\alpha_2(A)"] \\
				F_1''(h_A) \arrow[r, "\eta''(h_A)"] & F_2''(h_A)
			\end{tikzcd}
		\end{center}	
	\end{lemma}
	\begin{proof}
		{Direct verification.}
	\end{proof}
	\subsubsection{Presheaves and Sheaves on $(C^r/\cdot)$}\label{PShv-rel-Cr}
	We now specialize the discussion of \textsection\ref{abstract-nonsense} to the category $(C^r/\cdot)$. Consider the two ``forgetful" functors $t,b:(C^r/\cdot)\to (\text{Top})$ given on objects by
	\begin{align}
	    t : X/S &\mapsto X\\
	    b : X/S &\mapsto S
	\end{align}
	with the obvious action on morphisms. We have a natural transformation $s:t\Rightarrow b$ defined as follows. For any	$X/S$ in $(C^r/\cdot)$, the map $s(X/S)$ is the structure morphism $X\to S$. These admit left adjoints $t',b':(\text{Top})\to(C^r/\cdot)$ given on objects by
	\begin{align}
	    t' : S &\mapsto S/S\\
	    b' : S &\mapsto \emptyset/S
	\end{align}
	with the obvious action on morphisms. Here, $S/S$ denotes the trivial $S$-space with fibre a point, while $\emptyset/S$ denotes the trivial $S$-space with empty fibre. The adjunction isomorphism is obvious for both pairs $(t,t')$ and $(b,b')$. The induced natural transformation $s':b'\Rightarrow t'$ is also the obvious one.
	\begin{remark}\label{top-base}
	Suppose $\fX$ is a presheaf on the category $(C^r/\cdot)$.	Then, we have the presheaves $\fX^\text{top}:=\fX\circ t' = t''(\fX)$ and $\fX_\text{base}:=\fX\circ b' = b''(\fX)$ on	$(\text{Top})$ and the induced morphism $s''_\fX:\fX^\text{top}\Rightarrow\fX_\text{base}$ of presheaves. Lemma \ref{morphism-of-PShv} says that if $\fX = h_{X/S}$ is a Yoneda presheaf, then $s''_\fX$ corresponds	(under the Yoneda embedding) the structure map $X\to S$ of the topological spaces underlying the rel--$C^r$ manifold $X/S$. The construction of $\fX^\text{top}$, $\fX_\text{base}$ and	$s''_\fX$ is functorial and natural in $\fX$ (i.e., respects morphisms $\fX\to\fY$ of presheaves). Explicitly, any morphism $f:\fX\to\fY$ of presheaves on $(C^r/\cdot)$, yields the following commutative diagram
		\begin{center}
			\begin{tikzcd}
				\fX^\text{top} \arrow[d, "s''_\fX"] \arrow[r,"f^\text{top}"] & \fY^\text{top} \arrow[d,"s''_\fY"] \\
				\fX_\text{base} \arrow[r, "f_\text{base}"] & \fY_\text{base}
			\end{tikzcd}
		\end{center}
		of presheaves on $(\text{Top})$. We say that $f$ is a \textbf{rel--$C^r$ lift} of the continuous morphism $(f^\text{top},f_\text{base})$.
	\end{remark}
	\begin{definition}[Sheaves on $(C^r/\cdot)$]\label{sheaf-def}
		Suppose $\fX$ is a presheaf on $(C^r/\cdot)$. We then say that $\fX$ is \textbf{sheaf} if given any rel--$C^r$ manifold $p:X\to S$ and open covers $\{X_i\}_{i\in I}$ and $\{S_i\}_{i\in I}$ of $X$ and $S$ respectively, satisfying $p(X_i)\subset S_i$ the following two assertions are true.
		\begin{enumerate}[(1)]
			\item If $f,g\in\fX(X/S)$ are such that $f|_{X_i/S_i} = g|_{X_i/S_i}$ for all $i\in I$, then $f=g$.
			\item If $f_i\in\fX(X_i/S_i)$ for $i\in I$ are such that $f_i|_{X_{ij}/S_{ij}} = f_j|_{X_{ij}/S_{ij}}$ for all $i,j\in I$, then there exists $f\in\fX(X/S)$ such that we have $f|_{X_i/S_i} = f_i$ for all $i\in I$. Here, we set $X_{ij} := X_i\cap X_j$ and $S_{ij} := S_i\cap S_j$ for all $i,j\in I$
		\end{enumerate}
		Representable presheaves are clearly sheaves on $(C^r/\cdot)$. Notice that sheaves on $(\text{Top})$ can be defined entirely analogously and that $\fX^\text{top}$ and $\fX_\text{base}$ are sheaves if $\fX$ is a sheaf.
	\end{definition}
	\subsubsection{How to Represent a Sheaf on $(C^r/\cdot)$}\label{how-to-rep}
	\begin{definition}
	    Let $\fX$ be a sheaf on $(C^r/\cdot)$. 
	    \begin{enumerate}
	        \item(Open Sub-sheaf) A sub-sheaf $\fX'\subset\fX$ is called \textbf{open} if, for all rel--$C^r$ manifolds $X/S$ and an element $\varphi\in\fX(X/S)$, the fibre product $h_{X/S}\times_\fX\fX'$ (where the maps are given by $\varphi$ and the inclusion $\fX'\subset\fX$) is represented by a rel--$C^r$ manifold $U/T$ with the projection map 
	        \begin{align}
	            h_{U/T}\simeq h_{X/S}\times_\fX\fX'\to h_{X/S}
	        \end{align}
	        given by a morphism $U/T\to X/S$ consisting of open embeddings $U\to X$ and $T\to S$. (Note that we can also define an obvious notion of open sub-sheaf for sheaves on $(\text{Top})$ entirely analogously.)
	        \item (Base Change) Suppose $\Phi:\fS\to \fX_\text{base}$ is a morphism of sheaves on $(\text{Top})$. We then define the \textbf{base change of $\fX$ (along $\Phi$)} to be the sheaf $\fX_{\fS,\Phi}$ (or simply $\fX_{\fS}$ if $\Phi$ is clear from context) given by
	        \begin{align}
	            \fX_{\fS,\Phi}:=(\fS\circ b)\times_{(\fX_\text{base}\circ b)}\fX
	        \end{align}
	        where the maps used in the fibre product are $\Phi\circ b$ and $\fX\circ c$, where $c:b'\circ b\Rightarrow 1$ is the co-unit of the adjunction between $b'$ and $b$. The natural projection $\fX_{\fS,\Phi}\to\fS\circ b$ induces a natural isomorphism $(\fX_{\fS,\Phi})_\text{base}\xrightarrow{\simeq}\fS$. Moreover, the description of $\fX_{\fS,\Phi}$ as a fibre product yields a natural isomorphism
	        \begin{align}
	            (\fX_{\fS,\Phi})^\text{top}\xrightarrow{\simeq}
	            \fS\times_{\fX_\text{base}}\fX^\text{top}
	        \end{align}
	        with the projection $(\fX_{\fS,\Phi})^\text{top}\to\fS$ identified with the natural map $s''_{\fX_{\fS,\Phi}}$. In the case when $\Phi$ is the inclusion of an open sub-sheaf (or an open subset, when $\fS$ is representable), we denote the base change simply by $\fX|_\fS$. In this case, $\fX|_\fS\to\fX$ is the inclusion of an open sub-sheaf.
	        \item (Source Change) Suppose $\Psi:\fU\to\fX^\text{top}$ is the inclusion of an open sub-sheaf. We then define the \textbf{source change of $\fX$ (along $\Psi$)} to be the sub-sheaf $\fX^{\fU,\Psi}$ (or simply $\fX^\fU$) of $\fX$    given by setting $\fX^{\fU,\Psi}(Y/T)$ to be subset of $\fX(Y/T)$ consisting of elements $\chi$ such that the image $\chi^\text{top}\in \fX^\text{top}(Y)$ of $\chi$ can be lifted to $\fU(Y)$ along $\Psi(Y)$. This description gives a natural isomorphism
	        \begin{align}
	            (\fX^{\fU,\Psi})^\text{top}\xrightarrow{\simeq}\fU
	        \end{align}
	        which is compatible with $\Psi$ and the inclusion $\fX^{\fU,\Psi}\subset\fX$. The inclusion induces a natural isomorphism
	        \begin{align}
	            (\fX^{\fU,\Psi})_\text{base}\xrightarrow{\simeq}\fX_\text{base}.
	        \end{align}
	        Under these isomorphisms, $s''_\fX\circ\Psi$ is identified with $s''_{\fX^{\fU,\Psi}}$. Further, the sub-sheaf $\fX^{\fU,\Psi}\subset\fX$ is open.
	    \end{enumerate}
	\end{definition}
	\noindent
	The next lemma shows (a) the compatibility of the above definitions with representable functors and (b) the essentially local nature of representability for sheaves.
	\begin{lemma}[Properties of Representability]\label{prop-rep}
	    Representable sheaves have the following properties.
	    \begin{enumerate}[\normalfont (a)]
	        \item Let $X/S$ be an object of $(C^r/\cdot)$.
	        \begin{enumerate}[\normalfont(1)]
	            \item For any object $Y/T$ of $(C^r/\cdot)$, a morphism $h_{Y/T}\to h_{X/S}$ is equivalent to the inclusion of an open sub-sheaf if and only if the corresponding continuous maps $Y\to X$ and $T\to S$ are open embeddings.
	            \item Suppose $T\to S$ is a map of topological spaces. Then, the base change of $h_{X/S}$ under the corresponding map $h_T\to h_S$ is naturally isomorphic to $h_{X_T/T}$, where $X_T/T$ is the base change of $X/S$ along $T\to S$ in the sense of Definition \ref{rel-bc}.
	            \item Suppose $U\subset X$ is an open subset. Then, the source change $(h_{X/S})^{h_U}$ and the sheaf $h_{U/S}$ are equal as sub-sheaves of the sheaf $h_{X/S}$.
	        \end{enumerate}
	        \item Let $\fX$ be a sheaf on $(C^r/\cdot)$. Representability of $\fX$ is stable under base and source changes, and moreover, is local on $\fX$. More precisely, we have the following properties.
	        \begin{enumerate}[\normalfont(1)]
	            \item Suppose that $\{\fX_\alpha\to\fX\}_{\alpha\in I}$ is a cover by open sub-sheaves, i.e., for every $\varphi\in\fX(X/S)$, the collection $\{h_{X/S}\times_{\fX}\fX_\alpha\to h_{X/S}\}_{\alpha\in I}$ is represented by an open cover $\{X_\alpha/S_\alpha\}_{\alpha\in I}$ of $X/S$. In this case, $\fX$ is representable if and only if each $\fX_\alpha$ is representable.
	            \item If $\fX$ is representable and $\Phi:T\to\normalfont\fX_\text{base}$ is a morphism of topological spaces, then the base change $\fX_{T,\Phi}$ is also representable. More generally, if $\normalfont\fU\subset\fX^\text{top}$	is an open sub-sheaf and $\fX^\fU$ is representable, then so is the base change $(\fX^\fU)_{T,\Phi} = (\fX_{T,\Phi})^{\fU_T}$, where $\normalfont\fU_T\subset(\fX_{T,\Phi})^\text{top}$	is the open sub-sheaf obtained by pulling back $\fU$ under $\Phi$.
	        \end{enumerate}
	    \end{enumerate}
	\end{lemma}
	\begin{proof}
		\begin{enumerate}[(a)]
		    \item 
		        \begin{enumerate}[(1)]
		            \item Suppose first that $h_{Y/T}\to h_{X/S}$ is an open sub-sheaf. Then, by the definition of open sub-sheaf applied to $\varphi = \text{id}_{X/S}\in h_{X/S}(X/S)$, we see that $Y\to X$ and $T\to S$ are open embeddings. Conversely, if $Y\to X$ and $T\to S$ are open embeddings, then for any rel--$C^r$ map $(\tilde\varphi,\varphi):Z/B\to X/S$, we observe that $Z' = \tilde\varphi^{-1}(Y)\subset Z$ and $B' = \varphi^{-1}(T)\subset B$ are open embeddings. We also have the isomorphism
		            \begin{align}
		                h_{Z'/B'}\xrightarrow{\simeq} h_{Z/B}\times_{h_{X/S}}h_{Y/T}.
		            \end{align}
		            \item This assertion follows from the definition once we observe that the data of a rel--$C^r$ map $Z/B\to X_T/T$ is the same as the data of a rel--$C^r$ map $Z/B\to X/S$ and a continuous map $B\to T$ such that they both induce the same map $B\to S$.
		            \item This assertion also follows from the definition once we observe that the data of a rel--$C^r$ map $Z/B\to U/S$ is the same as the data of a rel--$C^r$ map $Z/B\to X/S$ such that the associated continuous map $Z\to X$ factors through the inclusion $U\subset X$. 
		        \end{enumerate} 
		        \item
		        \begin{enumerate}[(1)]
		            \item We prove that if each $\fX_\alpha$ is representable, then so is $\fX$ (the converse is easier). Identify each $\fX_\alpha$ with $h_{X_\alpha/S_\alpha}$ for some rel--$C^r$ manifold $X_\alpha/S_\alpha$. Since each $h_{X_\alpha/S_\alpha}\to\fX$ is equivalent to the inclusion of an open sub-sheaf, each pairwise fibre product $h_{X_\alpha/S_\alpha}\times_\fX h_{X_\beta/S_\beta}$ is represented by a rel--$C^r$ manifold $X_{\alpha\beta}/S_{\alpha\beta}$ with the maps $X_{\alpha\beta}/S_{\alpha\beta}\to X_\alpha/S_\alpha$ being rel--$C^r$ open embeddings. We may now define a rel--$C^r$ manifold $X/S$ by gluing $\{X_\alpha/S_\alpha\}_{\alpha\in I}$ along the `overlaps' $\{X_{\alpha\beta}/S_{\alpha\beta}\}_{\alpha,\beta\in I}$ (the relevant cocycle condition is also satisfied on triple overlaps). Now, since $\fX$ is a sheaf, the maps $h_{X_\alpha/S_\alpha}\to\fX$ glue together to a unique morphism $h_{X/S}\to\fX$. This is both a surjective map of sheaves and equivalent to the inclusion of an open sub-sheaf. As a result, it is an isomorphism.
		            \item This follows from (a)(2)--(3) above.
		        \end{enumerate}
		\end{enumerate}
	\end{proof}
	\noindent We can now state a basic criterion for representability of a presheaf on $(C^r/\cdot)$.
	\begin{proposition}[Representability Criterion]\label{rep-crit}
		Let $\fX$ be a sheaf on $(C^r/\cdot)$ and assume that the following conditions hold.
		\begin{enumerate}[\normalfont(1)]
			\item $\normalfont\fX_\text{base}$ is representable by a topological space $S$.
			\item There exists an open cover $\{U_\alpha\}_{\alpha\in I}$ of $S$ such that $\normalfont (\fX|_{U_\alpha})^\text{top}$ is representable. This implies that $\normalfont\fX^\text{top}$ is representable by a topological $S$-space $X$ with structure map $p:X\to S$.
			\item For every point $x\in X$, the sheaf \textbf{$\fX$ is representable near $x\in X$}, i.e., there exist		open neighborhoods $x\in U\subset X$ and $V\subset S$ with $p(U)\subset V$ such that $\fX^U|_V$ is representable.
		\end{enumerate}
		Then, $\fX$ is representable by a rel--$C^r$ structure on $X/S$.
	\end{proposition}
	\begin{proof}
		In view of assumptions (1) and (2), we get a map $p:X\to S$ of topological spaces representing $s''_\fX:\fX^\text{top}\to\fX_\text{base}$. Now, by assumption (3), we can find open covers $\{X_\alpha\}_{\alpha\in I}$ of $X$ and $\{S_\alpha\}_{\alpha\in I}$ of $S$ with $p(X_\alpha)\subset S_\alpha$ for all $\alpha\in I$ such that $\fX^{X_\alpha}|_{S_\alpha}$ is representable by a rel--$C^r$ structure on $X_\alpha/S_\alpha$. Now, Lemma \ref{prop-rep}(b)(1) shows that $\fX$ is representable by a rel--$C^r$ structure on $X/S$, since $\{h_{X_\alpha/S_\alpha}\to\fX\}_{\alpha\in I}$ gives an open cover of $\fX$ by representable sub-sheaves.
	\end{proof}
	\section{Moduli Problem}\label{mod-problem}
	We will now introduce the holomorphic curve moduli problem and then state the main theorem of this article. After that, we make some preliminary reductions before giving the key step of the proof in \textsection\ref{sc-ift}.
	\subsection{Moduli Functors}\label{mod-functor}
	Fix integers $g,m\ge 0$, a complex manifold $\cS$ and let $\pi:\cC\to\cS$ be a proper flat analytic family of prestable curves of genus $g$ with $m$ marked points (given by analytic sections $p_1,\ldots,p_m$ of $\pi$). For any point $s\in\cS$, let $(C_s,p_1(s),\ldots,p_m(s))$ denote the corresponding fibre of $\pi$, equipped with its $m$ marked points.
	\\\\
	Let $(X,J)$ be a smooth almost complex manifold and let $\beta\in H_2(M,\bZ)$ be a homology class. Denote by $\cC^\circ$ the (open) complement in $\cC$ of the union of the singular loci of the fibres of $\pi$ (in other words, $\cC^\circ$ is the locus where $\pi$ is an analytic submersion over $\cS$). Since $\cS$ is a complex manifold, it follows that $\cC^\circ$ is naturally a complex (and thus, smooth) manifold.	Let $E$ be a finite dimensional $\bR$-vector space and let $P$ be an $\bR$-linear homomorphism
	\begin{align}\label{inhom-perturbation}
		P:E\to C^\infty(\cC^\circ\times X, \Omega^{0,1}_{\cC^\circ/\cS}\boxtimes_{\bC} T_X).
	\end{align}
	We assume that that for every $v\in E$, the section $P(v)$ is supported on a subset of $\cC^\circ\times X$ which is proper over $\cS$ (with respect to the projection $\cC^\circ\times X\to\cC^\circ\xrightarrow{\pi}X$). We want to consider the moduli space of all \textbf{regular} triples $(s,u,v)$ where
	\begin{itemize}
		\item $s\in\cS$ is a point, $v\in E$ is a vector and $u:C_s\to X$ is a smooth map with $u_*[C_s]=\beta$.
		\item The map $u$ satisfies the equation
		\begin{align}\label{inhom-pde}
			\delbar_J(u) + P(v)(\cdot\,,u(\cdot)) = 0.
		\end{align}
	\end{itemize}
	The meaning of \textbf{regular} is that the linearization of (\ref{inhom-pde}) with respect to the pair of variables	$(u,v)$ is a surjective (Fredholm) operator. In the case  when $E=0$, we will also want to consider the sub moduli space consisting of those triples	$(s,u,v)$ for which $(C_s,p_1(s),\ldots,p_m(s),u)$ is stable. Both of these moduli spaces will be given canonical rel--$C^\infty$ structures by realizing them as solutions to the problem of representing	certain sheaves on the category $(C^\infty/\cdot)$.
	\begin{definition}[Moduli functor $\fM$]\label{def-mod-functor}
		We define the functor 
		\begin{align}
			\fM_{J,P}^\text{reg}(\pi,X)_\beta:(C^\infty/\cdot)^\text{op}\to(\text{Sets})
		\end{align}
		as follows. Given any topological space $T$ and a rel--$C^\infty$ manifold $q:Y\to T$, we define		$\fM_{J,P}^\text{reg}(\pi,X)_\beta(Y,q)$ to be the set of all diagrams
		\begin{center}
		\begin{tikzcd}
			\cC \arrow[d, "\pi"] &\cC_T \arrow[d] \arrow[l] &\cC_Y \arrow[d] \arrow[l] \arrow[r, "F"]&X \\
			\cS & T \arrow[l, "\varphi"] &  Y \arrow[l, "q"] \arrow[r, "w"] & E
		\end{tikzcd}
		\end{center}
		where $\varphi$ is continuous, both squares are Cartesian, $F$ is a rel--$C^\infty$ map from $\cC_Y/\cC_T$ to $X$ (where the rel--$C^\infty$ structure is pulled back from $Y/T$), $w$ is a rel--$C^\infty$ map from $Y/T$ to $E$ and finally, for every $y\in Y$, the triple $(s,u,v):=(\varphi(q(y)),F|_{C_s},w(y))$ satisfies		(\ref{inhom-pde}), is regular and lies in class $\beta$. (Here, we regard $X$ and $E$ as rel--$C^\infty$ manifolds over a point using Remark \ref{rel-over-pt}.) The ``restriction maps" of the presheaf $\fM_{J,P}^\text{reg}(\pi,X)_\beta$ are given, in an obvious way, by pullbacks.
	\end{definition}
	\begin{definition}[$\fM$ as a functor over $\cS$]
	    Viewing $\cS$ as a rel--$C^\infty$ $\cS$-manifold using the identity map (and identifying it with the Yoneda functor $h_{\cS/\cS}$ associated to it), we get a morphism (natural transformation) $\rho_\cS:\fM_{J,P}^\text{reg}(\pi,X)_\beta\to\cS$		given by mapping a diagram as above to the map $(\varphi\circ q,\varphi):Y/T\to\cS/\cS$.
	\end{definition}
	\begin{definition}[Moduli Functor $\Mbar$]
		In the case when $E=0$, we denote the presheaf $\fM^\text{reg}_{J,P}(\pi,X)_\beta$ simply by		$\fM^\text{reg}_J(\pi,X)_\beta$. In this case, we define the sub-presheaf $\Mbar^\text{reg}_J(\pi,X)_\beta\subset\fM^\text{reg}_J(\pi,X)_\beta$ by requiring that for every $y\in Y$, $(C_s,p_1(s),\ldots,p_m(s),F|_{C_s})$ is \textbf{stable}, where $s:=\varphi(q(y))$. Stability just means that the automorphism group of $(C_s,p_1(s),\ldots,p_m(s),F|_{C_s})$ is required to be finite.
	\end{definition}
	\begin{remark}
		It is evident from the definitions that $\fM_{J,P}^\text{reg}(\pi,X)_\beta$ and $\Mbar_{J}^\text{reg}(\pi,X)_\beta$ are sheaves on		$(C^\infty/\cdot)$.
	\end{remark}
	\noindent We can now state the key results of this article.
	\begin{theorem}\label{rep}
		The sheaf $\normalfont\fM^\text{reg}_{J,P}(\pi,X)_\beta$ is representable by a separated rel--$C^r$ manifold		over $\cS$ with $\rho_\cS$ representing the structure morphism.
	\end{theorem}
	\begin{theorem}\label{stable-rep}
		When $E=0$, the sub-sheaf $\normalfont\Mbar^\text{reg}_J(\pi,X)_\beta\subset		\fM^\text{reg}_J(\pi,X)_\beta$ is open and has the same base space $\cS$ and structure morphism $\rho_\cS$.
	\end{theorem}
	\begin{corollary}\label{stable-rep-cor}
		The sheaf $\normalfont\Mbar_J^\text{reg}(\pi,X)_\beta$ is representable by a separated rel--$C^r$ $\cS$-manifold with structure morphism given by $\rho_\cS$.
	\end{corollary}
	\begin{proof}
		{Direct consequence of Theorem \ref{stable-rep} and Theorem \ref{rep}.}
	\end{proof}
	\begin{remark} 
		In the sections that follow, we will abbreviate $\fM^\text{reg}_{J,P}(\pi,X)_\beta$ to $\fM$ (and similarly, $\Mbar^\text{reg}_J(\pi,X)_\beta$ to $\Mbar$) wherever possible -- provided $(X,J,\beta,\pi:\cC\to\cS)$ and $(E,P)$ are clear from the context. To prove Theorem \ref{rep}, we will check the three hypotheses of Proposition \ref{rep-crit} to show that $\fM$ is representable. We will then argue separately that $\Mbar\subset\fM$ is open, proving Theorem \ref{stable-rep} and Corollary \ref{stable-rep-cor}.
	\end{remark}
	\subsection{First Steps towards Representability}
	We develop some basic background material from topology and then verify the first two hypotheses of Proposition	\ref{rep-crit} for the functor $\fM$. We also establish the compatibility of the functor $\fM$ with base change -- this	will be useful when we establish local representability of $\fM$ in \textsection\ref{sc-ift}.
	\subsubsection{Primer on the $C^0$ topology}\label{c0-primer}
	Let $E,B,Z$ be metrizable topological spaces and let $\pi:E\to B$ be a continuous, proper and open map. Let $C^0(E/B,Z)$ be the set of pairs $(b,f)$ where $b\in B$ and $f:E_b\to Z$ is a continuous map (where we set $E_b:=\pi^{-1}(b)\subset E$).
	\begin{definition}[$C^0$ Topology]
		We define the \textbf{$C^0$ topology} on the set $C^0(E/B,Z)$ to be the topology generated by the following	basis of open sets. For any open $\cU\subset E\times Z$, we let
		\begin{align}
			\cN(\cU) := \{(b,f)\in C^0(E/B,Z)\;|\;\Gamma_f\subset\cU\}	
		\end{align}
		be a basic open set (here, $\Gamma_f\subset E_b\times Z\subset E\times Z$ denotes the graph of $f:E_b\to Z$).		We have 
		\begin{align}
			\cN(\cU_1)\cap\cN(\cU_2) = \cN(\cU_1\cap\cU_2)
		\end{align} 
		for any two open sets $\cU_1,\cU_2\subset E\times Z$, which shows that this is indeed a basis for a topology.		Sometimes, when $(b,f)\in\cN(\cU)$, we'll denote $\cN(\cU)$ as $\cN(f,\cU)$ to emphasize that it is an open neighborhood of $(b,f)$.
	\end{definition}
	\begin{lemma}\label{ev}
		The natural map $p:C^0(E/B,Z)\to B$ given by $(b,f)\mapsto b$ is continuous. In addition, the natural evaluation map
		\begin{align}
			\normalfont\text{ev}_{E/B,Z}:C^0(E/B,Z)\times_BE&\to Z\\
			((b,f),e)&\mapsto f(e)
		\end{align}
		is also continuous.
	\end{lemma}
	\begin{proof}
		Given an open set $U\subset B$, define $\cU_U:=\pi^{-1}(U)\times Z\subset E\times Z$, and note that we have $p^{-1}(U) = \cN(\cU_U)$. This proves the continuity of the map $p$. Now, given a point $((b,f),e)\in C^0(E/B,Z)\times_BE$, set $z:=f(e)$ and let $V\subset Z$ be a neighborhood of $z$. Choose an open set $e\in W\subset E$ such that $f(\overline{W}\cap E_b)\subset V$ (possible since $f$ is continuous and $E$ is metrizable -- we can simply take $W$ to be a small metric ball around $e$). Define $\cU_V$ to be the complement (in $E\times Z$) of the closed subset $\overline W\times(Z\setminus V)$. Clearly, $\Gamma_f\subset\cU_V$ and we have
		\begin{align}
			\cN(f,\cU_V)\times_BW\subset\text{ev}_{E/B,Z}^{-1}(V).
		\end{align}
		This proves the continuity of the map $\text{ev}_{E/B,Z}$.
	\end{proof}
	\begin{definition}[Hausdorff Metric]
		Let $d_E$ and $d_Z$ be metrics on $E$ and $Z$ inducing their respective topologies. This induces a metric		$d_{E\times Z}:=d_E+d_Z$ on the product $E\times Z$. Define the \textbf{induced Hausdorff metric} $d_H$ on $C^0(E/B,Z)$ to be the one for which the distance between two points $(b,f),(b',f')$ is the Hausdorff distance between the closed subsets $\Gamma_f,\Gamma_{f'}\subset E\times Z$ measured using $d_{E\times Z}$.
	\end{definition}
	\begin{lemma}\label{dH}
		The metric $d_H$ induces the $C^0$ topology on the set $C^0(E/B,Z)$.
	\end{lemma}
	\begin{proof}
		It's clear that the Hausdorff metric topology is finer than the $C^0$ topology (since $\epsilon$-neighborhoods with respect to $d_{E\times Z}$ of a graph $\Gamma_f$ form a fundamental system of neighborhoods in $E\times Z$ for the compact set $\Gamma_f$ as $\epsilon$ ranges over all positive real numbers).\\\\
		Now, we want to show that the $C^0$ topology is finer than the $d_H$-topology. To check this, let $\epsilon>0$	and $(b,f)\in C^0(E/B,Z)$ be given. Let $0<\delta\le\epsilon$ be such that for all $x,y\in E_b$ with $d_E(x,y)\le3\delta$, we have $d_Z(f(x),f(y))\le\epsilon$. Furthermore, let $x_1,\ldots,x_k$ be a maximal collection of points on $E_b$ such that $d_E(x_i,x_j)>\delta$ for $1\le i<j\le k$. Define
		\begin{align}
			U = \bigcap_{i=1}^k \pi(B_\delta(x_i,d_E)).
		\end{align}
		This is an open neighborhood of $b\in B$. Now, let $\cU$ be the (open) set of points $(e,z)\in E\times Z$ such that		$d_{E\times Z}((e,z),\Gamma_f)<\delta$ and $\pi(e)\in U$. Clearly, $\Gamma_f\subset\cU$ and thus, $\cN(f,\cU)$ is a neighborhood of $f$ in the $C^0$ topology. Now, for any $g\in N(f,\cU)$, we will show that $d_H(\Gamma_f,\Gamma_g)<5\epsilon$.\\\\
		For any $(e',g(e'))\in\Gamma_g$, we have $d((e',g(e')),\Gamma_f)\le\delta<5\epsilon$. Conversely, suppose we have $(e,f(e))\in\Gamma_f$. Pick $1\le i\le k$ such that $d_E(x_i,e)\le\delta$. Next, pick $e'\in E_{b'}$ such that $d_E(e',x_i)<\delta$ (this is possible since $b'\in U$). We can then find $(e'',f(e''))\in\Gamma_f$ such that
		\begin{align}
			d_E(e',e'') + d_Z(g(e'),f(e'')) < \delta.
		\end{align}
		We then have the estimates
		\begin{align}
			d_E(e,e')&\le d_E(e,x_i) + d_E(x_i,e') < 2\delta\le 2\epsilon\\
			d_E(e,e'')&\le d_E(e,x_i) + d_E(x_i,e') + d_E(e',e'') \le 3\delta\\
			d_Z(f(e),g(e')) &\le d_Z(f(e),f(e'')) + d_Z(f(e''),g(e')) \le \epsilon + \delta\le 2\epsilon
		\end{align}
		which together show that $d_{E\times Z}((e,f(e)),\Gamma_g)\le 4\epsilon < 5\epsilon$. Since $\epsilon>0$ and $(b,f)$ were arbitrary, this completes the proof that the $C^0$ topology is finer than the $d_H$-topology.
	\end{proof}
	\begin{lemma}[Universal Property of $C^0$ Topology]\label{c0-univ}
		Define the presheaf $\normalfont\fC^0(E/B,Z):(\text{Top})^\text{op}\to(\text{Sets})$ as follows. For a topological space $T$, a morphism $T\to\fC^0(E/B,Z)$, following the convention of Remark \ref{yoneda-lemma}, is defined to be a diagram
		\begin{center}
			\begin{tikzcd}
				E \arrow[d, "\pi"]& \arrow[l, "\tilde\varphi"]E_T \arrow[d, "\pi_T"] \arrow[r, "F"]& Z \\
				B & \arrow[l,"\varphi"] T &
			\end{tikzcd}
		\end{center}
		of continuous maps, where the square is commutative and Cartesian. (The pullback of a morphism $T\to\fC^0(E/B,Z)$ along a map $T'\to T$ is defined in the obvious way.) There's an obvious morphism of presheaves $\fC^0(E/B,Z)\to B$ given by forgetting everything in the above diagram except for $\varphi$, which makes $\fC^0(E/B,Z)$ a presheaf over $B$. Then, the natural evaluation map $\normalfont\text{ev}_{E/B,Z}$ (from Lemma \ref{ev}) gives a morphism
		\begin{align}
			C^0(E/B,Z)\to\fC^0(E/B,Z)
		\end{align} 
		over $B$ and this map an isomorphism of presheaves.
	\end{lemma}
	\begin{proof}
		Given a diagram as in the statement of the lemma (but with $F$ and $\varphi$ just being set maps), let $F_\varphi:T\to C^0(E/B,Z)$ be the set map defined by $t\mapsto (\varphi(t),F|_{E_{\varphi(t)}})$. It then suffices to show that $F_\varphi$ is continuous if and only if $F$ and $\varphi$ are.\\\\
		First, suppose $F_\varphi$ is continuous. Note that, using the notation of Lemma \ref{ev}, we have
		\begin{align}
			\varphi &= p\circ F_\varphi\\
			F &= \text{ev}_{E/B,Z}\circ(F_\varphi\times_B\text{id}_E)
		\end{align}
		which shows that $\varphi$ and $F$ are continuous. Conversely, suppose $\varphi$ and $F$ are continuous.		Let $t\in T$ and $(b_t,f_t):=F_\varphi(t)\in C^0(E/B,Z)$. Also, let $\cU\subset E\times Z$ be a given open neighborhood of $\Gamma_{f_t}$. Let $\cU_T'\subset E_T$ to be the pull back of $\cU$ under $(\tilde\varphi, F):E_T\to E\times Z$, and define
		\begin{align}
			V:=T\setminus\pi_T(E_T\setminus\cU_T').
		\end{align}
		By properness of $\pi_T$, we see that $V$ is an open neighborhood of $t\in T$. Moreover, $V$ satisfies
		\begin{align}
			F_\varphi(V)\subset\cN(f_t,\cU)
		\end{align}
		which completes the proof that $F_\varphi$ is continuous.
	\end{proof}
	\begin{corollary}[$C^0$ Topology Commutes with Pullbacks]\label{c0-bc}
		Let $T$ be a metrizable space, $\varphi:T\to B$ a continuous map and $\pi_T:E_T\to T$ to base change of 
		$\pi$ along $\varphi$. Then, the set map
		\begin{align}\label{c0-bc-iso}
			C^0(E_T/T,Z)&\to C^0(E/B,Z)\times_BT\\
			(t,f:{(E_T)}_t\to Z)&\mapsto((\varphi(t),f:E_{\varphi(t)}\to Z),t)
		\end{align}
		is a homeomorphism over $T$.
	\end{corollary}
	\begin{proof}
		This is a consequence of Lemma \ref{c0-univ}. Indeed, consider the corresponding $T$-morphism of sheaves
		\begin{align}
			\fC^0(E_T/T,Z)\to\fC^0(E/B,Z)\times_BT
		\end{align}
		which, for a topological space $S$, maps the diagram
		\begin{center}
			\begin{tikzcd}
				E_T \arrow[d, "\pi_T"]& \arrow[l, "\tilde\chi"]E_S \arrow[d, "\pi_S"] \arrow[r, "F"]& Z \\
				T & \arrow[l,"\chi"] S &
			\end{tikzcd}
		\end{center}
		to the pair consisting of the diagram
		\begin{center}
			\begin{tikzcd}
				E \arrow[d, "\pi"]& \arrow[l, "\tilde\varphi\circ\tilde\chi"]E_S \arrow[d, "\pi_S"] \arrow[r, "F"]& Z \\
				B & \arrow[l,"\varphi\circ\chi"] S &
			\end{tikzcd}
		\end{center}
		and the map $\chi:S\to T$. This is a bijection (natural in $S$) and so, \eqref{c0-bc-iso} is a homeomorphism.
	\end{proof}
	\subsubsection{Topology of $\fM$}
	Recall the definition of the moduli functor $\fM = \fM_{J,P}^\text{reg}(\pi,X)_\beta$ from Definition \ref{def-mod-functor}. We will now show that the functors $\fM^\text{top}$ and $\fM_\text{base}$, which are derived from $\fM$ (as in \textsection\ref{PShv-rel-Cr}) are representable using the results of \textsection\ref{c0-primer}. This will verify the first two hypotheses of the representability criterion (Proposition \ref{rep-crit}).
	\begin{lemma}\label{top-moduli}
		The functor $\normalfont\fM^\text{top}$ is represented by the subspace $\cZ\subset C^0(\cC/\cS,X)\times E$, endowed with the $C^0$ topology, consisting of regular $((s,f),v)$ satisfying (\ref{inhom-pde}) and $f_*[C_s]=\beta$. As a result, it is metrizable (by Lemma \ref{dH}) and in particular, Hausdorff.
	\end{lemma}
	\begin{proof}
		Define a morphism $\Psi:\fM^\text{top}\to\cZ$ of sheaves as follows. For any topological space $T$, $\Psi$ maps an element of $\fM^\text{top}(T) = \fM(T/T)$, given by the diagram
		\begin{center}
		\begin{tikzcd}
			\cC \arrow[d, "\pi"] &\cC_T \arrow[d] \arrow[l] &\cC_T \arrow[d] \arrow[l,"\text{id}"] \arrow[r, "F"]&X \\
			\cS & T \arrow[l, "\varphi"] &  T \arrow[l, "\text{id}"] \arrow[r, "w"] & E
		\end{tikzcd}
		\end{center}
		to the map
		\begin{align}
		    F_{\varphi,w}:T&\to\cZ\subset C^0(\cC/\cS,X)\times E \\
		    t&\mapsto ((\varphi(t),F|_{C_{\varphi(t)}}),w(t))
		\end{align}
		which is continuous by Lemma \ref{c0-univ}. Conversely, given any continuous map $T\to C^0(\cC/\cS,X)\times E$ with image contained in $\cZ$, we can use the universal property in Lemma \ref{c0-univ} to get a diagram as above belonging to $\fM^\text{top}(T)$. This defines the inverse $\cZ\to\fM^\text{top}$ to $\Psi$.
	\end{proof}
	\begin{lemma}\label{base-moduli}
		The functor $\normalfont\fM_\text{base}$ is represented by $\cS$ (the base of the family $\pi:\cC\to\cS$). The ``structure map" $\normalfont s''_\fM:\fM^\text{top}\to\fM_\text{base}$ is given by $((s,f),v)\mapsto s$
		and is separated.
	\end{lemma}
	\begin{proof}
        For a topological space $T$, an element of the set $\fM_\text{base}(T) = \fM(\emptyset/T)$ is, by definition, a diagram
		\begin{center}
		\begin{tikzcd}
			\cC \arrow[d, "\pi"] &\cC_T \arrow[d] \arrow[l] &\emptyset \arrow[d] \arrow[l] \arrow[r]&X \\
			\cS & T \arrow[l, "\varphi"] &  \emptyset \arrow[l] \arrow[r] & E
		\end{tikzcd}
		\end{center}
		which we can map to $\varphi\in h_\cS(T)$. This gives a natural isomorphism $\fM_\text{base}\xrightarrow{\simeq}h_\cS$. Note that under $s''_\fM$, an element of $\fM^\text{top}(T)$ given by a diagram as in the proof of Lemma \ref{top-moduli} maps to the element of $\fM_\text{base}(T)$ given by the above diagram. In other words, $s''_\fM$ maps $((\varphi,F),w)$ to $\varphi$. Finally, $s''_\fM$ is separated since $\fM^\text{top}$ is Hausdorff by Lemma \ref{top-moduli}.
	\end{proof}
	\begin{remark}
	    In view of Lemmas \ref{top-moduli} and \ref{base-moduli}, we will henceforth identify $\fM_\text{base}$ with $\cS$, $\fM^\text{top}$ with the subset of $C^0(\cC/\cS,X)\times E$ consisting of $((s,f),v)$ satisfying \eqref{inhom-pde} and $f_*[C_s]=\beta$ (and equipped with the $C^0$ topology) and $s''_\fM$ with the coordinate projection $((s,f),v)\mapsto s$.
	\end{remark}
	\subsubsection{Compatibility of $\fM$ with Base Change}
	\begin{lemma}\label{bc}
		Suppose $\Phi:\cS'\to\cS$ is an analytic map of complex manifolds. Let $\fM':=\normalfont\fM^\text{reg}_{J,P'}(\pi',X)_\beta$, where $\pi',P'$ are the pullbacks of $\pi,P$ along $\Phi$. Then, the natural map from $\fM'$to the base change $\fM_{\cS',\Phi}$ is an isomorphism of sheaves.
	\end{lemma}
	\begin{proof}
		{Direct verification. The proof is virtually identical to that of Corollary \ref{c0-bc}.}
	\end{proof}
	\begin{remark}\label{non-reg}
		Notice that Lemmas \ref{top-moduli}, \ref{base-moduli} and \ref{bc} have obvious analogues (which are still true)		even if we expanded the definition of $\fM$ to include non-regular triples $((s,f),v)$ satisfying (\ref{inhom-pde}) and $f_*[C_s]=\beta$. This observation will be useful in \textsection\ref{et-chart} (and also in formulating Lemma \ref{top-chart}). Of course, without restricting to regular triples, we could not have expected to represent $\fM$ by a rel--$C^\infty$ manifold.
	\end{remark}
	\section{Construction of $\fM$}\label{sc-ift}
	We will now show (using some inputs from polyfold theory) that $\fM$ is locally representable near each point of $\fM^\text{top}$. This will verify the third hypothesis of the representability criterion (Proposition \ref{rep-crit}) and thus will complete the proof of Theorem \ref{rep}.
	\subsection{Preparations}
	Let $(s,u,v_0)\in \fM^\text{top}$ be a point. We wish to show that $\fM$ is representable near $(s,u,v_0)$.\\\\
	\noindent For simplicity of notation, we will denote the curve $C_s$ as $C$. Let $\nu_1,\ldots,\nu_d$ be an enumeration of the nodes of the curve $C$ (where $d\ge 0$ is an integer). Define $N=\{u(\nu_i)\;|\;1\le i\le d\}$. Choose $h$ to be a Riemannian metric $X$ with the following property. There exist pairwise disjoint open subsets $\{U_\iota\}_{\iota\in N}$ in $X$ and a smooth function	$\varphi:U\to \bR^{2n}$ (where $2n = \dim X$ and $U := \coprod_\iota U_\iota$) such that the following hold:
	\begin{itemize}
		\item For each $\iota\in N$, we have $\iota\in U_\iota$ and $\varphi_\iota:=\varphi|_{U_\iota}$ is diffeomorphism		onto the unit ball $B^{2n}\subset\bR^{2n}$ centred at the origin with $\varphi_\iota(\iota)=0$.
		\item For each $\iota\in N$, the metric $h$ is standard in the coordinates $\varphi_\iota$, i.e.,
		\begin{align}
			h|_{U_\iota} = \varphi_\iota^*\left(\sum_{i=1}^{2n} dx_i^2\right).
		\end{align}
	\end{itemize}
	Usually, when speaking of points (or tangent vectors etc) in $U_\iota$, we will suppress $\varphi_\iota$ from the notation, and instead simply identify $U_\iota$ with $B^{2n}$ using $\varphi_\iota$. Next, for each $1\le j\le d$, there are two non-singular branches of $C$ meeting at $\nu_j$. Choose parametrizations $\psi_j:\bD\to C$ and $\psi_j':\bD\to C$ (with respective inverses $z_j$, $z_j'$ mapping the node $\nu_j$ to $0$). Let $\gamma:(0,\infty)\times S^1\to\bD^\times$ be the analytic isomorphism given by
	\begin{align}
		\gamma(s,t) = e^{-(s+it)}.
	\end{align}
	We call the two maps $\gamma_j,\gamma_j':(0,\infty)\times S^1\to C\setminus\{\nu_j\}$ as the $j^\text{th}$ pair of cylindrical parametrizations, where we set $\gamma_j:=\psi_j\circ\gamma$ and $\gamma_j':= \psi'_j\circ\gamma$. The $(0,\infty)\times S^1$-valued inverses of $\gamma_j$ and $\gamma_j'$ are denoted by $(s_j,t_j)$ and $(s_j',t_j')$ respectively and we call them the cylindrical coordinates on the $j^\text{th}$ pair of ends of $C\setminus\{\nu_1,\ldots,\nu_d\}$. Denote by $C'$ the complement (in $C$) of the union of the images of $\psi_j$ and $\psi_j'$ ($1\le j\le d$). We can, and shall, assume that the images of the maps $u\circ\psi_j$ (and $u\circ\psi_j'$) all lie in a compact subset of $U=\coprod_\iota U_\iota$.\\\\
	Having chosen a suitable metric and local coordinates on $X$, in addition to cylindrical parametrizations of the ends of $C\setminus\{\nu_1,\ldots,\nu_d\}$, we now proceed to fix coordinates on $\cS$ near the point $s$ (recall that $C$ is the fibre of $\pi$ over the point $s\in\cS$). We will do this by replacing $\pi$ by a particular versal deformation of $C$ which we describe in the following two paragraphs.
	\\\\
	Let $j_0$ denote the almost complex structure on $C$. Suppose that $V$ is a finite dimensional real vector space and let $\{j(v)\}_{v\in V}$ be a smooth family of almost complex structures on $C$ such that	$j(0)=j_0$ and $j(v)\equiv j_0$ over a neighborhood of the closure of the cylindrical ends. Also, let $G$ denote the space of all $\alpha = (\alpha_1,\ldots,\alpha_d)\in\bC^d$ with each $|\alpha_j|<1/4$.	Now, given $\alpha\in G$ and $v\in V$, define the prestable curve $(C_\alpha,j_\alpha(v))$ as follows. We first equip $C$ with the almost complex structure $j(v)$, and then we perform the	following cut-and-paste operations for each $1\le j\le d$.
	\begin{itemize}
		\item If $\alpha_j = 0$, we do not do anything.
		\item If $\alpha_j = e^{-(R_j + i\theta_j)}\ne 0$ for some (unique) point $(R_j,\theta_j)\in (0,\infty)\times S^1$, then we remove the node $\nu_j$ and the copies of $[R_j,\infty)\times S^1$ from either side of $\nu_j$, and then identify the remaining parts $(0,R_j)\times S^1$ of the ends according to the following equations
		\begin{align}
			s_j + s_j' &= R_j\\
			t_j + t_j' &= \theta_j.
		\end{align}
	\end{itemize}
	Note that these cut-and-paste operations can alternatively be described as replacing the locus $z_jz_j' = 0$ in $C$ with the locus $z_jz_j'=\alpha_j$ (with $|z_j|<1$ and $|z_j'|<1$). Also observe that $(C_0,j_0(0))$ is the same as the original curve $C$. By choosing $V$ and the family $\{j(v)\}_{v\in V}$ suitably, we can ensure that the following family is a versal deformation of the curve $C$.\\\\
	We take $\cS' := G\times V$ and define $\cC'$ by gluing the family $C'\times \cS'\to\cS'$ (where we use the almost complex structure $j(v)$ on $C'$ over a point $(\alpha,v)\in G$) along its	boundary (in the obvious way) with the family $q\times\text{id}_V:\cU_G\times V\to G\times V$, where we set
	\begin{align}
		\cU_G &:= \{(z_j,z_j')_{1\le j\le d}\in(\bar\bD^2)^d\;|\; (z_j\cdot z_j')_{1\le j\le d}\in G\}\\
		q(z_1,z_1',\ldots,z_d,z_d') &:= (z_1\cdot z_1',\ldots,z_d\cdot z_d').
	\end{align}
	Note that this family has the property that the fibre over $(\alpha,v)\in G\times V$ is given by $(C_\alpha,j_\alpha(v))$.
	\begin{lemma}[Reduction to the Versal Case]
	\label{reduction-versal}
	    We can choose $V$, $\{j(v)\}_{v\in V}$ and a map
	    \begin{align}
	        P':E\to C^\infty(\cC'^\circ\times X,\Omega^{0,1}_{\cC'^\circ/\cS'}\boxtimes_\bC T_X)
	    \end{align}
	    with support away from the nodes in $\pi':\cC'\to\cS'$ and proper over $\cS'$, following the notation of \eqref{inhom-perturbation}, so that the following assertion is true. If $\fM' = \fM^{\normalfont\text{reg}}_{J,P'}(\pi',X)_\beta$ is representable near $((0,0),u,v_0)$, in the sense of hypothesis $\normalfont\text{(3)}$ of Proposition \ref{rep-crit}, then $\fM=\fM^{\normalfont\text{reg}}_{J,P}(\pi,X)_\beta$ is also representable near $(s,u,v_0)$.
	\end{lemma}
	\begin{proof}
	    To begin with, we may choose $V$ and $\{j(v)\}_{v\in V}$ to be such that $\pi'$ is a versal deformation of $C$ at $(0,0)$. Thus, after shrinking $\cS$ around $s$, we get a (possibly non-unique) holomorphic classifying map $k:\cS\to\cS'$ and an isomorphism $\Phi:\cC\xrightarrow{\simeq}k^*\cC'$ with $k(s) = (0,0)$ and $\Phi|_{C_s} = \text{id}_{C}$. Changing the choice of $V$ and $\{j(v)\}_{v\in V}$ if necessary, we can assume that $k$ is an embedding. To see this, we argue as follows. After shrinking $\cS$ around $s$ if necessary, choose a local holomorphic coordinate chart $\psi:\cS\to W$, with $W$ a finite dimensional complex vector space and $\psi(s) = 0$. Now, define $V' = V\oplus W$ and consider the family obtained by pulling back $\pi'$ along the coordinate projection $p:G\times V'\to G\times V = \cS'$. Then, $p^*\cC'\to G\times V'$ is also a versal deformation of $C$ at $(0,0)$ and we can take the classifying map $k':\cS\to G\times V' = \cS'\times W$ given by
	    \begin{align}
	        k'(x) := (k(x),\psi(x)).
	    \end{align}
	    This map $k'$ is now an immersion near $s$ and thus, by shrinking $\cS$ and $\cS'$ if necessary, we may assume that $k'$ is a closed holomorphic embedding of $\cS$ into $\cS'$. (Now, replace $V,\cS',k$ by $V',G\times V',k'$ respectively.)\\\\
	    Now that we have embedded $\pi$ as a closed analytic subfamily of $\pi'$ via the maps $k:\cS\hookrightarrow\cS'$ and $\Phi:\cC\hookrightarrow\cC'$, we can extend the map $P$ of \eqref{inhom-perturbation} to a map $P'$, by using a cutoff function to ensure that its support stays proper over $\cS'$ and away from the nodes. With these definitions in place, let us assume that $\fM' = \fM^\text{reg}_{J,P'}(\pi',X)_\beta$ is representable near $((0,0),u,v_0)$, in the sense of hypothesis (3) of Proposition \ref{rep-crit}. Lemma \ref{bc}, combined with Lemma \ref{prop-rep}(b)(2), then shows that $\fM$ is also representable near $(s,u,v_0)$.
	\end{proof}
	\begin{remark}
	    We can (and do) shrink the cylindrical ends on $C$ to ensure that $P'(\cdot)$ is supported on the subset $C'\subset(C_\alpha,j_\alpha(v))$ for all $(\alpha,v)\in G\times V$. In view of Lemma \ref{reduction-versal}, we only need to prove the representability of $\fM^\text{reg}_{J,P'}(\pi',X)_\beta$ near $((0,0),u,v_0)$ in order to deduce the representability of $\fM$ near $(s,u,v_0)$. We will therefore replace $\pi:\cC\to\cS,s,P$ by $\pi':\cC'\to\cS',(0,0),P'$ for the remainder of \textsection\ref{sc-ift}.
	\end{remark}
	\begin{remark}
		For the rest of \textsection\ref{sc-ift}, we will use the coordinates $(\alpha,v)$ on the base of the family $\pi:\cC\to\cS$. Most of the statements we make about $(\alpha,v)\in\cS$ will be true only for $(\alpha,v)$ sufficiently close to the basepoint $(0,0)\in\cS$. 
		To enhance readability, we will avoid pointing this out explicitly in each instance.
	\end{remark}
	\noindent After a brief recollection of some relevant polyfold notions in \textsection\ref{polyfold-notation}, we will set up the problem of finding a chart for $\fM^\text{top}$ near $(s,u,v_0)$ in \textsection\ref{ift-chart} as an instance of the polyfold implicit function theorem from \cite{HWZ-ImpFuncThms}.
	\subsection{Interlude: Notational Conventions for Polyfolds}\label{polyfold-notation}
		For the discussion in \textsection\ref{polyfold-notation} and \textsection\ref{ift-chart}, we assume that the reader has a working knowledge of basic concepts of the theory of polyfolds developed in \cite{HWZ-Splicing,HWZ-ImpFuncThms,HWZ-Fred,HWZ-NewModels,HWZ-PFBasics,HWZ-GWbook}. See \cite[\textsection4--\textsection6]{FFGW-polyfold-survey} for an introduction and for further references. Here, we explain our conventions for gluing/anti-gluing maps and gluing profiles and how these differ (slightly) from
		what is customary in the polyfold literature.
		\subsubsection{Gluing ($\oplus,\hat\oplus$) and Anti-gluing ($\ominus,\hat\ominus$)}
		\label{gluing-map-def}
		Let's explain how our notations for gluing/anti-gluing differ from the standard polyfold literature. To this end, let $Y$ and $Y'$ be two copies of $(0,\infty)\times S^1$. As mentioned in the opening paragraph of \textsection\ref{ift-chart}, we take $S^1=\bR/2\pi\bZ$ (contrary to $S^1=\bR/\bZ$, which is customary in polyfold literature). We denote the obvious $\bR\times S^1$-valued coordinates on $Y$ and $Y'$ by $(s,t)$ and $(s',t')$ respectively. Given $\alpha\in\bC$ with $0<|\alpha|<1/4$ and writing $\alpha=e^{-(R+i\theta)}$, we define the infinite cylinder $Z_\alpha$ by identifying $Y$ and $Y'$ along $(0,R)\times S^1\subset Y$ and $(0,R)\times S^1\subset Y'$ using the relation
		\begin{align}
			\label{plumb1} s + s' &= R\\
			\label{plumb2} t + t' &= \theta.
		\end{align}
		We let $Y_\alpha$ be the subset of $Z_\alpha$ corresponding to the image of $(0,R)\times S^1\subset\Gamma$ in $Z_\alpha$. Notice that there are obvious $\bR\times S^1$-valued holomorphic coordinates $(s,t)$ and $(s',t')$ on $Y_\alpha$ (which also extend to $Z_\alpha$ in the		obvious way) which are related by equations (\ref{plumb1})--(\ref{plumb2}).\\\\ 
		Next, choose $\beta:\bR\to[0,1]$ to be a smooth function satisfying
		\begin{align}
			\label{cutoff1} \beta(s)\equiv 1&\text{   for }s\le -1\\
			\label{cutoff2} \beta(s) + \beta(-s) \equiv 1&\text{   for all }s\in\bR\\
			\label{cutoff3} -1\le\beta'(s)< 0&\text{   for all }s\in(-1,1).
		\end{align}
		For each $R>0$, define the function $\beta_R:\bR\to[0,1]$ to be the translate $\beta_R(s):=\beta(s-R/2)$. We will be using $\beta$ to define the \textbf{gluing} ($\oplus,\hat\oplus$) and \textbf{anti-gluing} ($\ominus,\hat\ominus$) maps below.\\\\
		Now, given continuous $\bR^N$-valued functions $\xi,\xi'$ on $Y,Y'$ respectively, we define the functions $\oplus_\alpha(\xi,\xi'),\ominus_\alpha(\xi,\xi')$ on		$Y_\alpha,Z_\alpha$ respectively by
		\begin{align}
			\oplus_\alpha(\xi,\xi')(s,t) &= \beta_R(s)\cdot\xi(s,t) + \beta_R(s')\cdot\xi'(s',t')\\
			\ominus_\alpha(\xi,\xi')(s,t) &= -\beta_R(s')\cdot(\xi(s,t) - [\xi,\xi']_R) + \beta_R(s)\cdot(\xi'(s',t')-[\xi,\xi']_R) 
		\end{align}
		where we set $[\xi,\xi']_R:=\frac12\left(\int_{S^1}\xi(\frac R2,t)\,\frac{dt}{2\pi} + \int_{S^1}\xi'(\frac R2,t')\,\frac{dt'}{2\pi}\right)$. Similarly, given continuous $\bR^N$-valued functions $\eta,\eta'$ on $Y,Y'$ 
		respectively, we define the functions $\hat\oplus_\alpha(\eta,\eta'),\hat\ominus_\alpha(\eta,\eta')$ on $Y_\alpha,Z_\alpha$ respectively by
		\begin{align}
			\hat\oplus_\alpha(\eta,\eta')(s,t) &= \beta_R(s)\cdot\eta(s,t) + \beta_R(s')\cdot\eta'(s',t')\\
			\hat\ominus_\alpha(\eta,\eta')(s,t) &= -\beta_R(s')\cdot\eta(s,t) + \beta_R(s)\cdot\eta'(s',t').
		\end{align}
		This definition deviates from the usual one in the polyfold literature in two ways. The first, less important, way is that we glue two copies of $(0,\infty)\times S^1$ to define $Y_\alpha,Z_\alpha$ rather than gluing $\bR_+\times S^1$ (with coordinates $s,t$) and $\bR_-\times S^1$ (with coordinates $s',t'$) using the relations $s-s'= R$ and $t-t'=\theta$. The second, more substantial, difference is that we make use of	the \textbf{logarithmic gluing profile} in our definition of gluing/anti-gluing as opposed to the \textbf{exponential gluing profile} (which is the one used in polyfold theory to guarantee sc-smoothness of the splicings/retractions induced by gluing/anti-gluing).		This point is explained in more detail below in \textsection\ref{gluing-profile-def}.
		\subsubsection{Gluing Profile (Logarithmic vs Exponential)}
		\label{gluing-profile-def}
		Central to defining the gluing/anti-gluing maps (as well as the domains $Z_\alpha$ and $Y_\alpha$ themselves) is the correspondence between $\alpha\in\bD$ and the parameters $(R,\theta)\in\bR_+\times S^1$ that go into defining the gluing relation between $(s,t)$ and $(s',t')$. Here, $\bD$ is the unit disc in $\bC$ centred at $0\in\bC$. The correspondence
		\begin{align}
			R &= -\log|\alpha|\\
			e^{-i\theta} &= \frac\alpha{|\alpha|}
		\end{align}
		which we used in the previous subsection is called the \textbf{logarithmic gluing profile}. On the other		hand, the \textbf{exponential gluing profile} is given by the correspondence $\alpha\mapsto (R_\text{exp},\theta_\text{exp})$, where
		\begin{align}
			R_\text{exp} &= 2\pi\cdot(e^{\frac1{|\alpha|}}-e)\\
			e^{-i\theta_\text{exp}} &= \frac\alpha{|\alpha|}.
		\end{align}
		The factor of $2\pi$ appears in the definition of $R_\text{exp}$ because in our convention $S^1=\bR/2\pi\bZ$	(instead of $\bR/\bZ$).
		\begin{definition}
			Define the map $\varphi_\text{exp}:\bD\to\bD$ to be the homeomorphism defined by $0\mapsto 0$ and
			\begin{align}
				0\ne\alpha\mapsto\varphi_\text{exp}(\alpha):=e^{-(R_\text{exp}+i\theta_\text{exp})}
			\end{align}
			where $R_\text{exp}$ and $\theta_\text{exp}$ are defined as above.
		\end{definition}
		\begin{lemma}
			$\normalfont\varphi_\text{exp}$ is a smooth map with $D^m\normalfont\varphi_\text{exp}(0)=0$ for all 
			$m\ge 0$.
		\end{lemma}
		\begin{proof}
			Smoothness away from $0$ is trivial (since $\varphi_\text{exp}$ even gives a diffeomorphism from	$\bD\setminus\{0\}$ to itself). Smoothness near $0$ is a straightforward consequence of the fact that $r\mapsto e^{-2\pi(e^{1/r}-e)}$ decays to $0$ very fast as $r\to 0$.
		\end{proof}
		\begin{definition}[Exponential Smooth Structure]\label{exp-smooth-structure}
			Let $G\subset\bD^d$ be an open neighborhood of $0\in\bD^d$ for some integer $d\ge 1$. We let $G_\text{exp}$ be the smooth manifold with underlying topological space $G$ and a global smooth chart $\psi_{G,\text{exp}}:G_\text{exp}\to \bC^d$ given by the formula
			\begin{align}
				(\alpha_1,\ldots,\alpha_d) \mapsto (\varphi_\text{exp}^{-1}(\alpha_1),\ldots,
				\varphi_\text{exp}^{-1}(\alpha_d)). 
			\end{align}
			Explicitly, this means that a function $f:G_\text{exp}\to\bR$ is smooth if and only if the function $f\circ\psi_{G,\text{exp}}^{-1}$ (defined on an open subset of $\bC^d$) is smooth in the usual sense. We will refer to $G_\text{exp}$ sometimes as \textbf{$G$ equipped with the exponential smooth structure}.
		\end{definition}
		\begin{corollary}\label{exp-id-smooth}
			The identity map $\normalfont G_\text{exp}\to G$ is smooth.
		\end{corollary}
		\begin{proof}
			Immediate from the definition of $G_\text{exp}$ and the smoothness of $\varphi_\text{exp}$.
		\end{proof}
		\begin{remark}
		    A perturbation term $P$ as in \eqref{inhom-perturbation} will be smooth with respect to the logarithmic (i.e. standard) smooth structure on the space $G$ of gluing parameters. Corollary \ref{exp-id-smooth} will allow us to conclude that $P$ is smooth with respect to $G_\text{exp}$ also. The reason we need to consider $G_\text{exp}$ is because the sc-smoothness results we need from polyfold theory work for the exponential smooth structure but not for the standard (logarithmic) smooth structure. 
		\end{remark}
	\subsection{Functional Analytic Setup and Implicit Function Theorem}\label{ift-chart}
	Fix a large integer $k$ (say, $k\ge 10$) and a constant $0<\delta<1$ and define the weighted Sobolev spaces
	\begin{align}
		&W^{k,2,\delta}(C,u^*T_X)\\
		&W^{k-1,2,\delta}(\tilde C,\Omega^{0,1}_{\tilde C,j_0}\otimes_\bC u^*T_X)
	\end{align}
	as in \cite[\textsection B.4]{Pardon-VFC}. For simplicity, we will refer to these weighted Sobolev spaces in the sequel as $H^{k,\delta}$ and $H^{k-1,\delta}$ respectively. (We caution the reader these same spaces are denoted by the slightly different notations $H^{k,\delta}_c(\cdots)$ and $H^{k-1,\delta}(\cdots)$ respectively in \cite{HWZ-GWbook}.	Following \cite{HWZ-GWbook}, we will view $H^{k,\delta}$ and $H^{k-1,\delta}$ as sc-Hilbert spaces in the usual way -- by picking a strictly increasing sequence $(\delta_m)_{m\ge 0}\subset(0,1)$ with $\delta=\delta_0$ and taking the $m^\text{th}$ scales to be $H^{k+m,\delta_m}$ and	$H^{k+m-1,\delta_m}$ respectively. Note that in \cite{HWZ-GWbook}, the authors take $S^1 = \bR/\bZ$, while	we take $S^1 := \bR/2\pi\bZ$. In our version, the first positive eigenvalue of $i\frac{\partial}{\partial\theta}$ on 
	$S^1$ is $1$.)
	\\\\
	Let $\Gamma(p,q)$, for sufficiently close by points $p,q\in X$, denote a smooth $J$-linear parallel transport from $q$ to $p$	(of a tangent vector) along the unique shortest $h$-geodesic joining the two of them. For example, we may take $\Gamma$ to be the parallel transport of the connection $\nabla^J:=\frac12(\nabla + J^{-1}\circ\nabla\circ J)$ on $T_X$, where $\nabla$ is the Levi-Civita connection of the metric $h$. For $v\in V$ small enough, we have the canonical isomorphism
	\begin{align}
		I(v) := \frac12(1 - j_0\circ j(v)):(T_{\tilde C},j(v))\to (T_{\tilde C},j_0)
	\end{align}
	of complex vector bundles. We denote by $I_\alpha(v):(T_{\tilde C_\alpha},j_\alpha(v))\to (T_{\tilde C_\alpha},j_\alpha(0))$ the corresponding isomorphism induced by $I(v)$ when it descends to the normalization $\tilde C_\alpha$ of $C_\alpha$ (which is constructed from $\tilde C$ by an obvious cut-and-paste operation). Note that $I_0(v) = I(v)$ and	that $I_\alpha(v)$ is an isomorphism for $v\in V$ small enough and is $\equiv1$ on the ends/necks of $\tilde C_\alpha$.
	\\\\For the rest of \textsection\ref{ift-chart}, we will assume the reader has a working knowledge of polyfold theory (e.g. at the level of \cite[\textsection4--\textsection6]{FFGW-polyfold-survey}). In some places, our notation is slightly different from what is standard in polyfold theory (these differences have been pointed out and explained in \textsection\ref{polyfold-notation}).
	\\\\
	Now, given $\alpha\in G$, $v\in V$, $\xi\in H^{k,\delta}$, all small enough, we define $\eta:=F(\alpha,v,\xi)\in H^{k-1,\delta}$ to be the unique section satisfying the following equations. (Note that on the cylindrical ends of $C$, $u^*T_X$ is trivialized by the chosen local Euclidean coordinates	$(U,\varphi)$ on $X$ and $\Omega^{0,1}_{\tilde C,j_0}$ is trivialized by $ds-i\cdot dt$. This enables us to regard $\xi,\eta$ as functions on the ends and thus, we can apply the gluing/anti-gluing maps to them.)
	\begin{align}
		\label{filled1}
		\Gamma(\oplus_\alpha\exp_u\xi,\oplus_\alpha u)\circ\hat\oplus_\alpha\eta\circ I_\alpha(v)
		&= \delbar_{J,j_\alpha(v)}\oplus_\alpha\exp_u\xi\\
		\label{filled2}
		\hat\ominus_\alpha\eta &= \delbar_{J(0)}\ominus_\alpha\xi.
	\end{align}
	We now clarify what the notation $\delbar_{J(0)}$ means. We can think of $\ominus_\alpha\xi$ as a collection of functions defined on copies of $\bR\times S^1$, one for each $1\le j\le d$ for which $\alpha_j\ne 0$. For any particular $1\le j\le d$ with $\alpha_j\ne 0$, we think of the function $\ominus_\alpha\xi$ restricted to the $j^\text{th}$ copy of $\bR\times S^1$ as taking values in	$T_{u(\nu_j)}X$ -- using the coordinate system $\varphi_\iota$ with $\iota := u(\nu_j)$ -- and on this vector function, $\delbar_{J(0)}$ acts as the usual Cauchy Riemann operator on the $\bC$-vector space $T_{u(\nu_j)}X$. (Equations (\ref{filled1})--(\ref{filled2}) are just equation (4.5) in \cite{HWZ-GWbook} rewritten slightly, following our conventions, as detailed in \textsection\ref{polyfold-notation}. Also, what's called	$\delta(\alpha,v)$ in \cite[Equation (4.5)]{HWZ-GWbook} is denoted as $I_\alpha(v)$ here.)	It is a fundamental fact that the resulting map
	\begin{align}\label{polyfold-delbar}
		F:G_\text{exp}\times V\times H^{k,\delta}\to H^{k-1,\delta}
	\end{align}
	is {\textbf{sc-smooth}}\footnote{\cite[Definition 4.17]{FFGW-polyfold-survey}} (here, $G_\text{exp}$ is just $G$ equipped with the exponential smooth structure -- see Definition \ref{exp-smooth-structure}) and is, in fact, also 
	{\textbf{sc-Fredholm}}\footnote{\cite[Definition 6.15]{FFGW-polyfold-survey}} (see \cite[Propositions 4.7, 4.8]{HWZ-GWbook} for proofs). Note that we are abusing notation slightly here as $F$ is defined only on a sufficiently small ball around $(0,0,0)\in G\times V\times H^{k,\delta}$. Next, given $e\in E$ and small enough $\alpha\in G$, $v\in V$, $\xi\in H^{k,\delta}$, we define $\eta := Q(e)(\alpha,v,\xi)\in H^{k-1,\delta}$ to be the unique section satisfying the following equation.
	\begin{align}\label{filled-perturb}
		\Gamma(\exp_u\xi,u)\circ\eta\circ I(v)
		 &= P(e)(\alpha,y,\cdot,\exp_u\xi(\cdot)).
	\end{align} 
	We explain the notation here. Recall that $P:E\to C^\infty(\cC^\circ\times\cS,\Omega^{0,1}_{\cC^\circ/\cS}	\boxtimes_\bC T_X)$ is an $\bR$-linear map, and we are evaluating it above at a point of $\cC^\circ\times X$. We have used the obvious parametrization $G\times V\times C'\times X\to \cC^\circ\times X$ (whose image contains the support of $P(e)$, for each $e\in E$). We can show (e.g. using	\cite[Theorem 13.5]{Palais-FNGA}) that the map $Q$ is in fact	classically smooth when viewed as a map $E\times G\times V\times H^{k,\delta_m}\to H^{k,\delta_{m+1}}$ where	for each $m\ge 0$, where $(\delta_m)_{m\ge 0}\subset(0,1)$ is the sequence of weights chosen to make $H^{k,\delta}$ an sc-Hilbert space. The weights don't really matter in the proof of this assertion since the sections $P(e)$ are supported away from the ends of the curve. (Crucially, we also observe that $Q$ remains smooth even if we replace $G$ by $G_\text{exp}$, since the identity map $G_\text{exp}\to G$ is smooth by Corollary \ref{exp-id-smooth}.) In particular, the resulting map
	\begin{align}
		Q : E\times G_\text{exp}\times V\times H^{k,\delta}\to H^{k-1,\delta} 
	\end{align}
	is an {\textbf{$\text{sc}^+$ map}}\footnote{\cite[Definition 6.10]{FFGW-polyfold-survey}}. Combining these two maps, we can define the map
	\begin{align}\label{non-linear-CR-section}
		\tilde F:(G_\text{exp}\times V)\times (E\times H^{k,\delta})&\to H^{k-1,\delta}\\
		(\alpha,v,e,\xi)&\mapsto F(\alpha,v,\xi) + Q(v_0+e)(\alpha,v,\xi)
	\end{align}
	which is again sc-smooth and sc-Fredholm (by the stability of sc-Fredholm sections under $\text{sc}^+$ perturbations, see \cite[Theorem 3.9]{HWZ-ImpFuncThms} or \cite[Theorem 6.21]{FFGW-polyfold-survey}). 
	\begin{remark}[Zero locus of $\tilde F$]
	    Note that a sufficiently small $(\alpha,v,e,\xi)$ is a zero of the map $\tilde F$ if and only if $((\alpha,v),\oplus_\alpha\exp_u\xi,v_0+e)$ satisfies (\ref{inhom-pde}) and $\ominus_\alpha\xi=0$. This follows from \eqref{filled1}, \eqref{filled2}, \eqref{filled-perturb} and the fact that $\delbar_{J(0)}\ominus_\alpha\xi = 0$ if and only if $\ominus_\alpha\xi = 0$ for $\xi\in H^{k,\delta}$.
	\end{remark}
	\noindent From the assumption that the point $(s,u,v_0)$, which corresponds to $(0,0,0,0)\in(G_\text{exp}\times V)\times (E\times H^{k,\delta})$, is regular (i.e., transversally cut-out), we deduce that the linearized operator
	\begin{align}
		D_0:=D_{3,4}\tilde F(0,0,0,0): E\oplus H^{k,\delta}\to H^{k-1,\delta}
	\end{align}
	is surjective. (Here, $D_{3,4}$ denotes the partial linearization with respect to the variables $e$ and $\xi$).	Define $K:=\ker D_0\subset E\oplus H^{k,\delta}$. We will now pick a sc-complement for $K$ in $E\oplus H^{k,\delta}$. The choice of this specific complement (which appeared earlier in \cite[\textsection B.7]{Pardon-VFC}) will be crucial to the proof of representability, so we explain it in detail.
	\begin{lemma}[Choice of sc-complement]\label{complement-props}
	    We can find a linear subspace $E'\subset E$, finitely many points $x_1,\ldots,x_\ell\in C'$ and linear subspaces $V_i\subset T_{u(x_i)}X$ for $i=1,\ldots,\ell$ such that the linear map
	    \begin{align}
		    \label{comp1} L:E\oplus H^{k,\delta}&\to E/E'\oplus\bigoplus_{1\le i\le\ell} (T_{u(x_i)}X)/V_i\\
		    \label{comp2} (e,\xi)&\mapsto (e\normalfont\text{ (mod }E'),\xi(x_1)\normalfont\text{ (mod }V_1),\ldots,\xi(x_\ell)\normalfont\text{ (mod }V_\ell))
	    \end{align}
	    restricts to an isomorphism on the linear subspace $K$. The subspace $\ker L\subset E\oplus H^{k,\delta}$ then defines an sc-complement of $K\subset E\oplus H^{k,\delta}$ and $D_0:\ker L\to H^{k-1,\delta}$ is an sc-isomorphism.
	\end{lemma}
	\begin{proof}
	    Let $E'$ be a complement to the image of the projection $K\to E$. Next, let $K'\subset K$ to be the subspace consisting of all elements of the form $(0,\xi)$. If $K'\ne 0$, then by unique continuation for elements in $K'$ (recall that they satisfy a homogeneous Cauchy-Riemann equation), we can find a point $x_1\in C'$ such that
	    \begin{align}
	        \text{ev}_{x_1}:K'&\to T_{u(x_1)}X \\
	        (0,\xi)&\mapsto \xi(x_1)
	    \end{align}
	    is nonzero. Choose $V_1\subset T_{u(x_1)}X$ to be a complement of the image of $\text{ev}_{x_1}$. Now, the subspace $K'_{x_1}\subset K'$ consisting of $\xi$ with $\xi(x_1)\in V_1$ satisfies $\dim K'_{x_1}<\dim K'$. If $K'_{x_1}\ne 0$, we repeat the argument to get $x_2\in C'$ and a subspace $V_2\subset T_{u(x_2)}X$ complementary to the (nonzero) image of $\text{ev}_{x_2}:K'_{x_1}\to T_{u(x_2)}X$. Continuing this way, this process must stop in $\ell\le\dim K'$ steps. In this way, we obtain the points $x_1,\ldots,x_\ell$ and the subspaces $V_i\subset T_{u(x_i)}X$ define the map $L$ as in \eqref{comp1}. We can now see that the map $L$ is bijective when restricted to $K\subset E\oplus H^{k,\delta}$ by arguing as follows. Define a descending filtration on the target of \eqref{comp1} by declaring the $i^\text{th}$ filtered piece to be the subspace consisting of the elements with the first $i$ entries equal to zero (for $0\le i\le k$). Pulling this back via $L$ gives filtration on $K$ as well. Now, it suffices to observe that $L|_K$ is a filtered map whose associated graded map is an isomorphism. Moreover, we can regard $L$ as an sc-linear surjection and thus, $\ker L$ is an sc-complement of $K$ and $D_0|_{\ker L}$ is an sc-isomorphism.
	\end{proof}
    \noindent By the sc-implicit function theorem (more precisely, by \cite[Theorem 4.6]{HWZ-ImpFuncThms}), we have an sc-smooth (and in particular, classically smooth -- since $K$ is finite dimensional) map 
	\begin{align}
		\sigma:U_u\subset G_\text{exp}\times V\times K\to\ker L
	\end{align}
	defined on an open neighborhood $U_u$ of $(0,0,0)\in G_\text{exp}\times V\times K$ and a constant $\epsilon>0$, which together satisfy the following properties.
	\begin{itemize}
		\item We have $\sigma(0,0,0)=0$.
		\item The set $Z$ of points $(\alpha,v,e,\xi)\in G_\text{exp}\times V\times E\times H^{k,\delta}$ satisfying	the conditions
		\begin{align}
			\max\{\|\xi\|_{k,\delta},\|e\|,\|\alpha\|,\|v\|\}<\epsilon\\ 
			\tilde F(\alpha,v,e,\xi)=0
		\end{align} 
		is identical to the graph of the map $\sigma:U_u\to\ker L$.
		\item For $\kappa\in K$, $\alpha\in G$ and $v\in V$ with $(\alpha,v,\kappa)\in U_u$, let's denote $\sigma(\alpha,v,\kappa)\in E\oplus H^{k,\delta}$ as $(e_{\alpha,v,\kappa},\xi_{\alpha,v,\kappa})$.		Similarly, denote $\kappa\in E\oplus H^{k,\delta}$ as $(\kappa_E,\kappa_H)$. Then, for all $(\alpha,v,\kappa)\in U_u$, the partial linearization 
		\begin{align}
			D_{\alpha,v,\kappa}:=D_{3,4}\tilde F(\alpha,v,
			\kappa_E+e_{\alpha,v,\kappa},\kappa_H + \xi_{\alpha,v,\kappa}):E\oplus H^{k,\delta}\to H^{k-1,\delta}
		\end{align}
		is a surjective Fredholm map and the projection onto $K = \ker D_0$ (along the chosen sc-complement $\ker L$) 
		induces a linear isomorphism
		\begin{align}\label{kernel-gluing}
			\ker D_{\alpha,v,\kappa}\xrightarrow{\simeq}\ker D_0.
		\end{align}
	\end{itemize}
	\begin{remark}
	    The map $\sigma$ can easily be constructed using standard gluing analysis as in \cite[Appendix B]{Pardon-VFC}, but its smoothness properties are not very evident from this construction. However, by carrying out standard gluing analysis more carefully, it is possible to establish the smoothness of $\sigma$ relative to $G$ (e.g. \cite[Lemma B.8.1]{Pardon-VFC} helps us show that $\sigma$ is ``rel--$C^{1,1}$"). We have instead chosen to take the easier route of directly getting the sc-smoothness of $\sigma$ by invoking \cite[Theorem 4.6]{HWZ-ImpFuncThms}.
	\end{remark}
	We can now define a chart for the moduli space near	the point $(s,u,v_0)$.
	\begin{definition}[Gluing Chart]
		Define the set map
		\begin{align}
			\varphi_{(s,u,v_0)}:U_u&\to\fM^\text{top}\\
			(\alpha,v,\kappa)&\mapsto ((\alpha,v),\oplus_\alpha\exp_u(\kappa_H + 
			\xi_{\alpha,v,\kappa}) ,v_0+\kappa_E+e_{\alpha,v,\kappa}).
		\end{align}
		Clearly, the map $\varphi_{(s,u,v_0)}$ commutes with the projection to $\cS$ (i.e., the projection to $(\alpha,v)$ on the source and the projection $\rho_\cS$ on the target). The isomorphism (\ref{kernel-gluing}) ensures that the triple $\varphi_{(s,u,v_0)}(\alpha,v,\kappa)$ is indeed regular.
	\end{definition}
	\noindent By shrinking $U_u$, we may assume that $U_u$ is a product neighborhood of $(0,0,0)$ of the form $G_u\times V_u\times K_u$, where $G_u,V_u,K_u$ are $\epsilon'$-balls centred at $0$ in $G,V,K$ respectively (for some sufficiently small number $\epsilon'>0$). For notational simplicity in the next definition, let's replace the family $\pi:\cC\to\cS$ by the family $\pi:\cC|_{G_u\times V_u}\to G_u\times V_u$. Thus, in this new notation, we have $\cS = G_u\times V_u$ and $U_u = \cS\times K_u$.
	\begin{definition}[Local Universal Family]
		The \textbf{universal evaluation map} $\text{ev}_{(s,u,v_0)}$ associated to the map $\varphi_{(s,u,v_0)}$ fits into the following diagram
		\begin{center}
		\begin{tikzcd}
			\cC\times K_u \arrow[d, "\pi\times\text{id}_{K_u}"]\arrow[r, "\text{ev}_{(s,u,v_0)}"]& X\\
			U_u &
		\end{tikzcd}
		\end{center}
		and is defined as follows. Given $(z,\kappa)\in \cC\times K_u$, let $(\alpha,v):=\pi(z)\in\cS$. Then, we get the map 
		\begin{align}
			u_{\alpha,v,\kappa}:=\oplus_\alpha\exp_u(\kappa_H+\xi_{\alpha,v,\kappa}):C_\alpha\to X
		\end{align}
		from $(\alpha,v,\kappa)\in U_u$, and we define $\text{ev}_{(s,u,v_0)}(z,\kappa):=u_{\alpha,v,\kappa}(z)$.	Similarly, we define the \textbf{universal obstruction map} $\fo_{(s,u,v_0)}$ associated to $\varphi_{(s,u,v_0)}$ by
		\begin{align}
			\fo_{(s,u,v_0)}:U_u&\to E\\
			(\alpha,v,\kappa)&\mapsto v_0+\kappa_E + e_{\alpha,v,\kappa}.
		\end{align}
	\end{definition}
	\begin{theorem}\label{smooth-chart-map}
		The following commutative diagram (with both squares Cartesian)
		\begin{center}
		\begin{tikzcd}
			\cC \arrow[d,"\pi"]&\arrow[l,"="]\cC \arrow[d,"\pi"]&
			\arrow[l] \cC\times K_u \arrow[d, "\pi\times\text{id}_{K_u}"]\arrow[r, "\normalfont
			\text{ev}_{(s,u,v_0)}"]& X\\
			\cS&\arrow[l,"="]\cS&\arrow[l] U_u \arrow[r,"\fo_{(s,u,v_0)}"]& E
		\end{tikzcd}
		\end{center}
		where the unlabeled arrows are the coordinate projections, defines an element
		\begin{align}
			\normalfont\tilde{\varphi}_{(s,u,v_0)}\in\fM(U_u/\cS)
		\end{align}
		which is a rel--$C^\infty$ lift of the continuous $\cS$-morphism $\normalfont\varphi_{(s,u,v_0)}:		U_u\to\fM^\text{top}$, in the sense of Remark \ref{top-base}. 
	\end{theorem}
	\begin{proof}
		We need to prove that the map $\text{ev}_{(s,u,v_0)}$ is rel--$C^\infty$ on $\cC\times K_u$ (relative to the projection onto $\cC$) and that the map $\fo_{(s,u,v_0)}$ is rel--$C^\infty$ on $U_u$	(relative to the projection onto $\cS$). The latter assertion is evident as a consequence of the smoothness of the map $\sigma:U_u\to E\oplus H^{k,\delta}$ (where we use the exponential smooth structure on $G$). For the former assertion, let's begin by noting that the map $U_u\to H^{k,\delta}$ given by $(\alpha,v,\kappa)\mapsto\xi_{\alpha,v,\kappa}$ is smooth (again, when we use the exponential smooth structure on $G$). In the language of Appendix \ref{wkrel}, this implies that the map
		\begin{align}
			e_1:\cC\times K_u\to \cC\times H^{k,\delta}\\
			(z,\kappa)\mapsto(z,\kappa_H + \xi_{\alpha,v,\kappa})
		\end{align}
		where $(\alpha,v):=\pi(z)$, is a weak rel--$C^\infty$ $\cC$-morphism. Define the map
		\begin{align}
			\label{e2} e_2:\cC\times H^{k,\delta}\to\cC\times X\\
			(z,\xi)\mapsto (\oplus_\alpha\exp_u\xi)(z)
		\end{align}
		where again we set $(\alpha,v):=\pi(z)$. This is also a weak rel--$C^\infty$ $\cC$-morphism by Lemma \ref{wkrel-ev}. Thus, the map $e_2\circ e_1$ is also a weak rel--$C^\infty$ $\cC$-morphism by Remark \ref{wk-properties}(3). This implies, by Remark \ref{wk-properties}(2), that $\text{ev}_{(s,u,v_0)}$ is		rel--$C^\infty$ relative to the projection $\cC\times K_u\to\cC$, as desired.
	\end{proof}
	\begin{lemma}[Topological Chart]\label{top-chart}
		$\varphi_{(s,u,v_0)}$ maps a neighborhood of $(0,0,0)\in U_u$ homeomorphically onto a neighborhood in the $C^0$ topology of $(s,u,v_0)$ in the space of all triples $(s',u',v')$ solving (\ref{inhom-pde}), including the non-regular maps. In particular, regular maps are open among all maps solving (\ref{inhom-pde}).
	\end{lemma}
	\begin{proof}
		We will deduce the first statement from \cite[Lemma B.12.1]{Pardon-VFC} by checking the following.
		\begin{itemize}
			\item $\varphi_{(s,u,v_0)}$ is continuous and injective when restricted to a neighborhood of $(0,0,0)\in U_u$.	Indeed, continuity of $\varphi_{(s,u,v_0)}$ follows from Theorem \ref{smooth-chart-map}. For injectivity, we argue as in \cite[Lemma B.10.3]{Pardon-VFC} to see that $\varphi_{(s,u,v_0)}(\alpha,v,\kappa) = \varphi_{(s,u,v_0)}(\alpha',v',\kappa')$ implies that $(\alpha,v)=(\alpha',v')$ and (provided $\kappa, \kappa'$ are small enough) that $\kappa + \sigma(\alpha,v,\kappa) = \kappa' + \sigma(\alpha',v',\kappa')$. Applying $L$ to both sides, this implies that $\kappa=\kappa'$.
			\item The $C^0$ topology on the target is metrizable and therefore, in particular, Hausdorff.
			\item Any neighborhood of $(0,0,0)\in U_u$ maps onto a (not necessarily open) neighborhood of $(s,u,v_0)$		in the $C^0$ topology. This is proved by using the exponential decay estimate (\cite[Proposition B.11.1]{Pardon-VFC}) for the derivatives of pseudoholomorphic curves on long cylinders mapping into small balls. See \cite[Proposition B.11.5]{Pardon-VFC} for the details of this argument.
		\end{itemize}
		Finally, to deduce the second statement, we note that every point in the image of $\varphi_{(s,u,v_0)}$ is regular by the isomorphisms in (\ref{kernel-gluing}).
	\end{proof}
	\noindent In view of this Lemma, we can (and do) replace $U_u$ by a smaller neighborhood of $(0,0,0)$ to assume that the inverse of $\varphi_{(s,u,v_0)}$ is an $\cS$-chart for $\fM^\text{top}$ near $(s,u,v_0)$. By shrinking $K_u, G_u, V_u$, we can still assume that $U_u = G_u\times V_u\times K_u$.	
	\subsection{Proof that $\fM$ is Locally Representable}
	Define the open subset $U_{(s,u,v_0)}\subset\fM^\text{top}$ to be the homeomorphic image of $U_u$ under $\varphi_{(s,u,v_0)}$.
	\begin{theorem}\label{loc-rep}
		The natural transformation
		\begin{align}\label{smooth-chart-map2}
			\normalfont\tilde\varphi_{(s,u,v_0)} : U_u/\cS \to \fM^{U_{(s,u,v_0)}}
		\end{align}
		coming from Theorem \ref{smooth-chart-map} is an isomorphism of contravariant functors from $(C^\infty/\cdot)$ to $\normalfont(\text{Sets})$. Moreover, this isomorphism respects the structure maps to $\cS$.
	\end{theorem}
	\begin{proof}
		We will construct the inverse of the natural transformation (\ref{smooth-chart-map2}) and leave the verification that it is indeed natural and an inverse to the reader. Thus, let us start with a diagram
		\begin{center}
		\begin{tikzcd}
			\cC \arrow[d, "\pi"] &\cC_T \arrow[d] \arrow[l] &\cC_Y \arrow[d] \arrow[l] \arrow[r, "F"]&X \\
			\cS & T \arrow[l, "\varphi"] &  Y \arrow[l, "q"] \arrow[r, "w"] & E
		\end{tikzcd}
		\end{center}
		such that the induced (continuous) map $F_\varphi:Y\to\fM^\text{reg}_{J,P}(\pi,X)_\beta^\text{top}$ has image $\subset U_{(s,u,v_0)}$. Thus, we obtain a continuous relative morphism $(\Phi,\varphi):Y/T\to U_u/\cS$ with $\Phi:=\varphi_{(s,u,v_0)}^{-1}\circ F_\varphi$.	We shall now show that $\Phi$ is rel--$C^\infty$ to conclude the definition of the inverse natural transformation of (\ref{smooth-chart-map2}).\\\\
		At this point, we ask the reader to revisit the definition of the complement of $K = \ker D_0$	in $E\oplus H^{k,\delta}$. Specifically, we need to recall equations (\ref{comp1}) and (\ref{comp2}), which involve the points $x_1,\ldots,x_\ell\in C'$ and the subspaces $E'\subset E, V_1\subset T_{u(x_1)}X,\ldots,V_\ell\subset T_{u(x_\ell)}X$. \\\\
		Define the function $\Psi:Y\to K$ as follows. Given $y\in Y$, let $x_i(y)\in\cC_Y$ be the copy of the point $x_i$ in the fibre of $\cC_Y\to Y$ over $y$ (for $i=1,\ldots,\ell$). Now, set
		\begin{align}\label{chart-transition}
			\Psi(y):=L^{-1}\left(w(y)-v_0\text{ (mod }E'),\exp_{u(x_1)}^{-1}F(x_1(y))\text{ (mod }V_1),\ldots,
			\exp_{u(x_\ell)}^{-1}F(x_\ell(y))\text{ (mod }V_\ell)\right) \in K.
		\end{align}
		It is easy to see that $y\mapsto x_i(y)$ defines a rel--$C^\infty$ map $Y/T\to\cC_Y/\cC_T$ for each $1\le i\le\ell$, and thus, it follows that $\Psi$ is a rel-$C^\infty$ morphism $Y/T\to K$. We will identify		$\Psi$ and $\Phi$ (i.e., showing $\Phi = (\varphi\circ q,\Psi):Y\to \cS\times K$) which will complete the proof.	Thus, let $y\in Y$ be given. Set $(\alpha,v,\kappa):=\Phi(y)$. We then know that the restriction of $F$ to the fibre over $y$ (of $\cC_Y\to Y$) and $w(y)\in E$ can be identified with 
		\begin{align}
			u_{\alpha,v,\kappa}:=\oplus_\alpha\exp_u(\kappa_H &+ \xi_{\alpha,v,\kappa}):C_\alpha\to X\\
			v_0 + \kappa_E + e_{\alpha,v,\kappa}&\in E.
		\end{align}
		Thus, we have $\exp_{u(x_i)}(\kappa_H(x_i) + \xi_{\alpha,v,\kappa}(x_i)) = F(x_i(y))$ for $1\le i\le\ell$. But, recall that $L(e_{\alpha,v,\kappa},\xi_{\alpha,v,\kappa})=0$ which means that $\xi_{\alpha,v,\kappa}(x_i)\in V_i$ (for $1\le i\le\ell$) and $e_{\alpha,v,\kappa}\in E'$. Combining this with (\ref{chart-transition}), we obtain
		\begin{align}
			L(\kappa) = (L\circ\Psi)(y)
		\end{align}
		and thus, $\kappa = \Psi(y)$, as desired. This completes the proof that $\Phi$ is rel--$C^\infty$.
	\end{proof}
	\subsection{Completing the Proof of Theorem \ref{rep}}
	We use the representability criterion (Proposition \ref{rep-crit}), to argue that $\fM$ is representable. Indeed, Lemma \ref{base-moduli}, Lemma \ref{top-moduli} and Theorem \ref{loc-rep} verify the hypotheses of Proposition \ref{rep-crit}, thus completing the proof of Theorem \ref{rep}.
	\section{Moduli Spaces with Obstruction Bundles}\label{mod-with-ob}
	For some applications of moduli spaces of holomorphic curves,	we also need to consider certain naturally defined vector bundles on them (e.g. ``obstruction bundles"). We show how to construct rel--$C^\infty$ structures on such vector bundles, using only the universal property from Theorem \ref{rep}.
	\subsection{Relative Tangent Bundles}
	\begin{definition}[Rel--$C^r$ Vector Bundles]
		Let $X/S$ be a rel--$C^r$ manifold and suppose $E$ is a topological $\bR$-vector bundle on $X$ of rank $n$. Given two local trivializations of $E$, over open sets $U,V\subset X$, we say that they are \textbf{rel--$C^k$ compatible} for some $0\le k\le r$ if and only if the transition map $U\cap V\to GL(n,\bR)$ is rel--$C^k$ (with source considered as an $S$-manifold). A \textbf{rel--$C^k$ structure on $E$} is a maximal atlas of local trivializations of $E$, which are pairwise rel--$C^k$ compatible. A vector bundle $E$ equipped with a fixed rel--$C^k$ structure will be called a rel--$C^k$ vector bundle over $X/S$. Analogous definitions can be made		for $\bC$-vector bundles.
	\end{definition}
	\begin{remark}
		The total space of a rel--$C^k$ vector bundle $E$ (on $X/S$) naturally admits the structure of a rel--$C^k$ manifold over $S$ (with the structure map being the projection $E\to X\to S$). It's also easy to see that the same is true of associated bundles of $E$, such as $E^*$, $\Lambda^m E$ and $E\otimes V$ (for $V$ a fixed finite dimensional vector space). In particular, in this way, $E\oplus E$ is a rel--$C^k$ manifold over $S$. Moreover, the addition, scalar multiplication and zero section maps
		\begin{align}
			E\oplus E\to E,&\quad (v,w)\mapsto v+w\\
			\bR\times E\to E,&\quad (\lambda,v)\mapsto \lambda v\\
			X\to E,&\quad x\mapsto 0\in E_x
		\end{align}
		are all $S$-morphisms of class rel--$C^k$. We also observe that rel--$C^k$ vector bundles can be pulled back		along morphisms of class rel--$C^m$ (for $k\le m$) to yield a rel--$C^k$ structure on the pullback bundle.
	\end{remark}
	\begin{definition}[Relative Tangent Bundles]
		Suppose $X/S$ is a rel--$C^r$ manifold for some $r\ge 1$. We define the \textbf{relative tangent bundle of $X/S$}, denoted by $T_{X/S}$, to be the rel--$C^{r-1}$ vector bundle on $X/S$ defined as follows.	For a local chart $\cU\subset V_S$ (with $V$ a vector space), we define the relative tangent bundle to be $(T_V)_S|_\cU$ and for a chart transition $\Phi:V_S\to W_S$, we glue these by the differential map $T\Phi:(T_V)_S\to (T_W)_S$. By gluing together these definitions over all the charts for $X/S$ in its maximal atlas, we obtain the rel--$C^{r-1}$ vector bundle $T_{X/S}$ over $X/S$.
	\end{definition}
	\begin{remark}
		Given a rel--$C^k$ map $\varphi:X'/S'\to X/S$ between rel--$C^r$ manifolds (for some $1\le k\le r$), we get		an induced linear bundle map (of class rel--$C^{k-1}$) called the differential
		\begin{align}
			T\varphi: T_{X'/S'}\to \varphi^*T_{X/S}
		\end{align}
		of the map $\varphi$, defined in the obvious way using local charts. Clearly, this differential satisfies		the chain rule. We will sometimes denote $T\varphi$ as either $T_\varphi$ or $d_{X'/S'}\varphi$.
	\end{remark}
	\noindent For the rest of the discussion, we exclusively focus on rel--$C^\infty$ manifolds and rel--$C^\infty$ vector bundles on them.
	\begin{definition}[Relative Submersion]
		A rel--$C^\infty$ $S$-morphism $\varphi:Y'/S\to Y/S$ is called a \textbf{relative submersion} if and only if		$T_\varphi$ is a surjective bundle homomorphism. It is a simple consequence of the implicit function theorem (with continuous dependence on parameters) that if $\varphi$ is a submersion and $Z/S\to Y/S$ is any other rel--$C^\infty$ $S$-morphism, then $(Z\times_YY')/S$ is naturally a rel--$C^\infty$ manifold.
	\end{definition}
	\begin{lemma}
	    If $\varphi:Y'/S\to Y/S$ is a rel--$C^\infty$ relative submersion, then $Y'/Y$ naturally acquires the structure of a rel--$C^\infty$ manifold and the natural sequence of vector bundles
		\begin{align}
			0\to T_{Y'/Y}\to T_{Y'/S}\xrightarrow{T_\varphi} \varphi^*T_{Y/S}\to 0
		\end{align}
		is exact on $Y'$.
	\end{lemma}
	\begin{proof}
	    The question is local on $Y'$ and so, we may assume that $Y'$ is an open neighborhood of $(0,0,s_0)\in V\times W\times S$ while $Y$ is an open neighborhood of $(0,s_0)\in W\times S$ for some vector spaces $V,W$ with $\varphi$ given by
	    \begin{align}
	        \varphi(x,y,s) = (f(x,y,s),s)
	    \end{align}
	    with $f(0,0,s_0)=0$ and the derivative of $f$ at $(0,0,s_0)$ in the $W$ direction given by $\text{id}_W$. Consider the rel--$C^\infty$ map $\Phi:(x,y,s)\mapsto(x,f(x,y,s),s)$, which, by the inverse function theorem with parameters (see Lemma \ref{rel-ift} below), admits a local rel--$C^\infty$ inverse. Changing to these new coordinates, we may safely assume that $\varphi$ is simply given by the projection $(x,y,s)\mapsto(y,s)$. In this local model, the result is apparent.
	\end{proof}
	\begin{lemma}
	    If $\varphi:Y'/S\to Y/S$ and $Z/S\to Y/S$ are rel--$C^\infty$ $S$-morphisms with $\varphi$ a relative submersion, then $(Z\times_YY')/S$ is naturally a rel--$C^\infty$ manifold.
	\end{lemma}
	\begin{proof}
	    The question is local on $Y'$. Arguing as in the proof of the previous lemma, we may assume that $Y' = Y\times M$ for some $C^\infty$ manifold $M$, with $\varphi$ given by the coordinate projection $Y\times M\to Y$. In this case, we can see that $Z\times_YY' = Z\times M$ has a tautological rel--$C^\infty$ structure over $S$.
	\end{proof}
	\begin{lemma}[Local Triviality of Families of Vector Spaces]\label{loc-triv}
		Let $\pi: E/S\to X/S$ be a relative submersion with fibre dimension $n$. Suppose 
		we have rel--$C^\infty$ $S$-morphisms
		\begin{align}
			a:E\times_XE\to E\\
			s:\bR\times E\to E\\
			\fz:X\to E
		\end{align}
		which are also $X$-morphisms. Assume that for each $x\in X$, the fibre $E_x:=\pi^{-1}(x)$ carries the structure of an $n$-dimensional $\bR$-vector space with the restrictions of $a,s,\fz$ being the addition, scalar multiplication and zero vector respectively. Then, given any $x\in X$, there is an open neighborhood $x\in U\subset X$ such that we have an $X$-morphism $E|_U\to U\times\bR^n$ which is a rel--$C^\infty$ $S$-diffeomorphism and intertwines the addition, scalar multiplication and zero vector on either side. In particular, $E$ naturally carries the structure of a rel--$C^\infty$ $\bR$-vector bundle of rank $n$ on $X/S$.	The obvious analogue of for $\bC$-vector bundles also holds.
	\end{lemma}
	\begin{proof}
		Consider the rel--$C^\infty$ vector bundle $E':=\fz^*T_{E/X}$ on $X/S$. We will produce an $X$-morphism $E\xrightarrow{\Psi} E'$ which is a rel--$C^\infty$ $S$-diffeomorphism and is linear on each fibre of $\pi:E\to X$. This will suffice to prove the statement of the lemma. We define
		\begin{align}
			\Psi: v\in E_x\mapsto \frac{\partial}{\partial t}s(t,e)|_{(0,v)}\in T_{\fz(x)}E_x.
		\end{align}
		It is possible (using the axioms relating $a,s,\fz$) to see that for each $x\in X$, $\Psi|_{E_x}:E_x\to T_{\fz(x)}E_x$	is linear. In fact, Lemma \ref{lie-gp-vector} shows that $\Psi$ is a bijection (and that we can identify the differential of $\Phi|_{E_x}$ at $\fz(x)$ with the identity map of $T_{\fz(x)}E_x$ for each $x\in X$). Since $\Psi$ is manifestly rel--$C^\infty$, it will be enough (by the inverse function theorem with continuous dependence on parameters; see Lemma \ref{rel-ift}) to show that the map
		\begin{align}
			T_{E/S}\xrightarrow{T_{\Psi}}\Psi^*T_{E'/S}
		\end{align}
		is an isomorphism. It will be enough to check this along the zero section $\fz(X)\subset E$. This is because	the assertion will then follow (by continuity) for a neighborhood of the zero section and then for all of $E$	(since $\Psi$ commutes with scaling). Let's now consider $\fz^*T_{\Psi}:\fz^*T_{E/S}\to\fz'^*T_{E'/S}$ where		$\fz' := \Psi\circ\fz$ is the zero section of $E'$. Both the domain and the target of this map split into the tangent space along the fibre and the tangent space along the zero section. It's immediate to see that		$\fz^*T_\Psi$ is the identity map on each of those summands. This implies that $\fz^*T_\Psi$ is an isomorphism.
		\\\\
		To deal with the case of $\bC$-vector bundles, we simply observe that a rel--$C^\infty$ $\bR$-vector bundle with a rel--$C^\infty$ fibrewise almost complex structure $J:E\to E$ admits rel--$C^\infty$ local $\bC$-linear trivializations.
	\end{proof}
	\begin{lemma}
		\label{lie-gp-vector}
		Suppose $V$ is a $C^\infty$ $n$-dimensional manifold and suppose that the $C^\infty$ maps
		\begin{align}
			a : V\times V\to V\\
			s :\bR\times V\to V\\
			\fz\in V
		\end{align}
		are the addition, scalar multiplication and zero vector of an $n$-dimensional $\bR$-vector space structure on 
		$V$. Then, the map
		\begin{align}
			\Psi:V&\to T_\fz V\\
			v&\mapsto \frac{\partial}{\partial t}s(t,e)|_{(0,v)}
		\end{align} 
		is a smooth $\bR$-linear isomorphism. Moreover, its derivative at $\fz\in V$ is the identity map of $T_\fz V$. 
	\end{lemma}
	\begin{proof}
		We note that $(V,a,\fz)$ is an abelian Lie group under the given assumptions. The map $s$ gives a contraction of $V$ to $\fz$ and thus, $V$ is contractible. By general Lie group theory, this means that that the exponential map		$\exp:T_\fz V\to V$ is a diffeomorphism (and even a group isomorphism since $V$ is abelian) with $d(\exp)_\fz$ being the identity map. Thus, it will be enough to show that $\exp\circ\Psi$ is the identity map of $V$ and that $\Psi$ respects scaling. The latter is proved by the computation
		\begin{align}
			\Psi(s(\lambda,v)) = \frac{d}{dt}s(t,s(\lambda,v))|_{t=0}
			= \frac{d}{dt}s(\lambda t,v))|_{t=0} = \lambda\cdot\Psi(v).
		\end{align}
		For the former, we argue as follows. Given $v\in V$, we claim that $t\mapsto \gamma_v(t):=s(t,v)$ is the same as the exponential trajectory $t\mapsto \exp(t\cdot\Psi(v))$. To see this, we note that $\gamma_v(0) = s(0,v) = \fz$ and that
		\begin{align}
			\gamma_v'(\tau) = \frac{d}{dt}s(t+\tau,v)|_{t=0} = \frac{d}{dt}a(s(t,v),\gamma_v(\tau))|_{t=0} = 
			da(\cdot,\gamma_v(\tau))|_0\cdot\Psi(v)
		\end{align}
		which shows that $s(t,v) = \exp(t\cdot \Psi(v))$. Evaluating at $t=1$ gives $(\exp\circ\Psi)(v)=v$.
	\end{proof}
	\begin{lemma}[Inverse Function Theorem with Parameters]\label{rel-ift}
		Suppose $V,W$ are finite dimensional $\bR$-vector spaces, $(S,t)$ is a pointed topological space and $\cU\subset V\times S$ is an open neighborhood of $(0,t)$. Let $f:\cU\to W$ be a rel--$C^1$ map (with respect to the projection $\cU\to S$) given by $(q,s)\mapsto f_s(q)$ such that we have $f_t(0)=0$ and the derivative $df_t(0):V\to W$ is an isomorphism. Then, there exist open neighborhoods $U\subset\cU$ and $U'\subset W\times S$ of $(0,t)$ such that the map $F:\cU\to W\times S$ given by $(q,s)\mapsto(f_s(q),s)$ is a rel--$C^1$ $S$-diffeomorphism	$U\xrightarrow{\simeq} U'$, i.e., it is bijective and its inverse $G$ is also rel--$C^1$. Moreover, if $2\le r\le\infty$ and $f$ is rel--$C^r$, then so is $G$.
	\end{lemma}
	\begin{proof}
		(The proof is essentially the same as that of the inverse function theorem.) Without loss of generality, we may assume that $V=W$ and $df_t(0) = \text{id}_V$. Endow $V$ with a norm $\|\cdot\|$. Now, we may assume (by shrinking $\cU$ and $S$) that $\|df_s(q)-\text{id}_V\|\le\frac12$ for all $(q,s)\in\cU$ and that $\cU=N\times S$ where $0\in N\subset V$ is the $\epsilon$-ball centred at $0$ for some $\epsilon>0$. Now, define $h_s(q):=f_s(q)-q$ and notice that by the mean value inequality, we have
		\begin{align}
			\|f_s(q_1)-f_s(q_2)\|\ge\|q_1-q_2\|-\|h_s(q_1)-h_s(q_2)\|\ge\frac12\|q_1-q_2\|
		\end{align}
		which shows that $f|_\cU$ is injective. Now, given $q'\in N$ (sufficiently small) and $s\in S$, we would like to		find $q\in N$ such that $q'=f_s(q)=q+h_s(q)$, i.e., $q = q' - h_s(q)$. In order to get a solution $q$ by the contraction mapping principle, we would like to find $\delta>0$ and further shrink $S$ so that for all $\|q\|\le\delta$, $\|q'\|\le\delta/3$ and $s\in S$, we have $\|q'-h_s(q)\|\le\delta$. We estimate
		\begin{align}
			\|q'-h_s(q)\|\le\|q'\| + \|h_s(q)-h_s(0)\| + \|h_s(0)\| &\le \delta/3 + \delta/2 + \|h_s(0)\|
		\end{align}
		and thus, after shrinking $S$ to assume that $\|h_s(0)\| = \|h_s(0)-h_t(0)\|\le\delta/6$ for all $s\in S$, we conclude	that for all $\|q'\|\le\delta/3$ and $s\in S$, there exists a unique $q = g_s(q')$ with $\|q\|\le\delta$ such that $f_s(q)=q'$. From the explicit constructive proof of \cite[Theorem 9.23]{Rudin}, we deduce the estimate
		\begin{align}
		    \|g_s(q') - g_{s_0}(q_0')\|\le 2\|q'-f_s(g_{s_0}(q_0'))\|    
		\end{align}
		for all $s,s_0\in S$ and $\|q'\|,\|q'_0\|\le\delta/3$ which shows that the map $(s,q')\mapsto g_s(q')$ is continuous at each point $(s_0,q_0')$. By shrinking $\cU$ to a suitable open $U$ and $W\times S$ to a suitable open $U'$, we thus see that $F$ is a homeomorphism. We can show, by the proof of the inverse function theorem (see \cite[Theorem 9.24]{Rudin}), that the function $q'\mapsto g_s(q')$ is differentiable for each $s\in S$, with derivative given by the formula
		\begin{align}\label{ift-derivative}
			dg_s(q') = \left[df_s(g_s(q'))\right]^{-1}
		\end{align}
		which completes the proof that $g(s,q'):=g_s(q')$ is rel--$C^1$ on $U'/S$. Now, we may use \eqref{ift-derivative} and induction on $r$, using the chain rule from Lemma \ref{rel-chain} and the smoothness of the map $A\mapsto A^{-1}$ on $GL(V)$, to conclude that $g$ is rel--$C^r$ if $f$ is rel--$C^r$.
	\end{proof}
	\subsection{Obstruction Bundles}
	We specialize now to the situation of \textsection\ref{mod-problem}, additionally assuming that $E=0$. We let $\fM$ denote the moduli space of pairs $(s,f)$ with $s\in\cS$ and $f:C_s\to X$ a $J$-holomorphic map which	is regular (i.e., cut-out transversely) and satisfies $f_*[C_s]=\beta$. Assume that we have a smooth $\bC$-vector bundle $F\to X$ equipped with a $\bC$-linear connection $\nabla$. Given any smooth map $f:C\to X$, with $C$ a prestable	curve, the pullback $f^*F$ naturally acquires the structure of a holomorphic vector bundle when equipped with the $\bC$-linear Cauchy-Riemann operator $(f^*\nabla)^{0,1}$.
	\begin{definition}
		We say that a $J$-holomorphic map $f:C\to X$ is \textbf{$F$-unobstructed} provided it is regular as a holomorphic map in $X$ (i.e., is cut-out transversely by $\delbar_J$) and, in addition, satisfies $H^1(C,f^*F)=0$.
	\end{definition}
	\begin{definition}
		We define $\fM_F\subset\fM$ to be the sub-sheaf of pairs $(s,f)$, where $s\in\cS$ and $f:C_s\to X$ is $J$-holomorphic and $F$-unobstructed.
	\end{definition}
	\begin{remark}
	    Note that over $\fM_F$, we have a family of $\bC$-vector spaces (of constant rank) given by 
	    \begin{align}
	        (s,f)\mapsto H^0(C,f^*F):=\ker (f^*\nabla)^{0,1}.
	    \end{align}
	   It is natural to ask if this can be realized canonically as a rel--$C^\infty$ $\bC$-vector bundle over $\fM_F$, and Theorem \ref{obs-thm} below answers this is in the affirmative.
	\end{remark}
	\noindent Now, note that the tangent bundle of $F$ sits in the canonical short exact sequence
	\begin{align}\label{TF}
		0\to \pi^*F\to T_F\xrightarrow{d\pi} \pi^*T_X\to 0.
	\end{align}
	The connection $\nabla$ gives a splitting $T_F\simeq \pi^*F\oplus\pi^*T_X = \pi^*(F\oplus T_X)$ and thus, defines an almost complex structure $J_F$ on the manifold $F$ by the pulling back the complex vector bundle structure from $F\oplus T_X$ (where the almost complex structure on the second factor is $J$). Note that that with this definition, the projection map $\pi:(F,J_F)\to (X,J)$ is pseudo-holomorphic. Furthermore, a $J_F$-holomorphic	map $\tilde f:C\to F$ is exactly the data of a $J$-holomorphic map $f:C\to X$ (given by $f = \pi\circ\tilde f$) and a smooth section $\sigma\in \ker (f^*\nabla)^{0,1}$.
	\begin{lemma}\label{F-unob}
		Let $C$ be a prestable curve and $\tilde f:C\to F$ be $J_F$-holomorphic and set $f:=\pi\circ\tilde f$. If $f$ is $F$-unobstructed, then $\tilde f$ is regular. Conversely, if $\tilde f$ is regular, then so is $f$. Moreover, if $\tilde f$ has image lying in the zero section of $E$, then we may additionally conclude that $f$ is $F$-unobstructed.
	\end{lemma}
	\begin{proof}
		We pullback the sequence (\ref{TF}) along $\tilde f$ to get the sequence
		\begin{align}
			0\to f^*F\to \tilde f^*T_F\to f^*T_X\to 0
		\end{align}
		which then leads to the following commutative diagram with exact rows
		\begin{center}
		\begin{tikzcd}
			0 \arrow[r] & \Omega^0(C,f^*F) \arrow[d, "(f^*\nabla)^{0,1}"] \arrow[r]& 
			\Omega^0(C,\tilde f^*T_F) \arrow[d, "D(\delbar_{J_F})_{\tilde f}"] \arrow[r]& 
			\Omega^0(C,f^*T_X) \arrow[d, "D(\delbar_J)_f"] \arrow[r]& 0\\
			0 \arrow[r] & \Omega^{0,1}(\tilde C,f^*F) \arrow[r] & \Omega^{0,1}(\tilde C,\tilde f^*T_F) \arrow[r] & 
			\Omega^{0,1}(\tilde C,f^*T_X) \arrow[r]& 0
		\end{tikzcd}
		\end{center}
		where $\tilde C$ is the normalization of $C$. Now, using the snake lemma gives the required result. One only needs to notice that when $\tilde f$ has image lying in the zero section of $E$, the short exact sequence of complexes above splits and thus, the connecting morphism in the long exact sequence is zero.
	\end{proof}
	\begin{corollary}
		$\fM_F$ is represented by an open subset of $\fM$.
	\end{corollary}
	\begin{proof}
		Given any $J$-holomorphic map $f:C_s\to X$, composing with the zero section $X\to F$ gives a $J_F$-holomorphic	map $\tilde f:C_s\to F$. This provides a map from the space of pairs $(s,f)$ to the space of pairs $(s,\tilde f)$. This map is continuous since it can be defined at the level of the corresponding functors. Taking the inverse image of the regular locus under this map gives precisely $\fM_F\subset\fM$. The regular locus is open by Lemma \ref{top-chart} and thus, so is $\fM_F\subset\fM$.
	\end{proof}
	\begin{definition}
		We define $\fF$ to be the moduli space of pairs $(s,\tilde f)$ where $s\in\cS$ and $\tilde f:C_s\to F$ is		$J_F$-holomorphic, regular and satisfies $\tilde f_*[C_s]=\beta$. This can be defined as a sheaf on $(C^\infty/\cdot)$ just as in Definition \ref{def-mod-functor} and is thus representable by Theorem \ref{rep}. Let $\pi_*:\fF\to\fM$ be the map given by $(s,\tilde f)\mapsto (s,f):=(s,\pi\circ\tilde f)$. It is a rel--$C^\infty$ $\cS$-morphism since it can be defined at the level of functors. Furthermore, we define the open subset $\fF_F := \pi_*^{-1}(\fM_F)\subset\fF$.
	\end{definition}
	\begin{remark}
		Notice that the fibre of $\pi_*:\fF_F\to\fM_F$ over a point $(s,f)$ can be identified canonically with the	$\bC$-vector space $H^0(C,f^*F)$. Notice that by the assumption of $F$-unobstructedness and the Riemann-Roch Theorem, it follows that $\dim_{\bC} H^0(C,f^*F) = N(1-g) + \langle c_1(F),\beta\rangle$, where $N$ and $g$ are respectively the rank of $F$ and the arithmetic genus of $C_s$. Thus, we may view $\pi_*$ as describing a family of complex vector spaces.
	\end{remark}
	\begin{theorem}[Obstruction Bundle on Moduli Space]\label{obs-thm}
		The map 
		\begin{align}
			\pi_*:\fF_F\to\fM_F
		\end{align} 
		whose fibres are canonically complex vector spaces, defines a rel--$C^\infty$ $\bC$-vector bundle on $\fM_F/\cS$ of rank given by $N(1-g) + \langle c_1(F),\beta\rangle$.
	\end{theorem}
	\begin{proof}
		Let us first prove that $\pi_*$ is a submersion. This is a pointwise condition, and thus, we only need to check it		after base changing to each point $s\in \cS$. The base changes $(\fF_F)_s$ and $(\fM_F)_s$ are smooth manifolds		and the differential $T_{\pi_*}|_{(\fF_F)_s}$ of the map $\pi_*$ at a point $\tilde f\in(\fF_F)_s$ is given by
		\begin{align}
			\ker D(\delbar_{J_F})_{\tilde f} \xrightarrow{\tilde f^*d\pi} \ker D(\delbar_J)_f
		\end{align}
		where $f:=\pi\circ\tilde f$. This map is surjective using the snake lemma on the commutative diagram from the proof of Lemma \ref{F-unob} in conjunction with the assumption that $H^1(C_s,f^*F)=0$. Next, note that the fibrewise addition, scalar multiplication and zero section maps of $F$
		\begin{align}
			a:F\times_X F\to F\\
			s:\bC\times F\to F\\
			\fz:X\to F
		\end{align}
		give rise to the rel--$C^\infty$ $\cS$-morphisms (which are also $\fM_F$-morphisms)
		\begin{align}
			a_*:\fF_F\times_{\fM_F}\fF_F\to \fF_F\\
			s_*:\bC\times \fF_F\to\fF_F\\
			\fz_*:\fM_F\to\fF_F
		\end{align}
		which give the fibrewise vector space structures on $\fF_F$. These maps are rel--$C^\infty$ because they		can be defined in an obvious fashion as maps of functors. Now, by Lemma \ref{loc-triv}, it follows that		$\pi_*:\fF_F\to\fM_F$ possesses rel--$C^\infty$ $\bC$-linear local trivializations over $\fM_F/\cS$.
	\end{proof}
	\section{Construction of $\Mbar$}\label{stable-mod}
	We will now prove Theorem \ref{stable-rep}. To do this, we need some standard facts about holomorphic curves (elliptic regularity, Gromov-Schwarz Lemma) reformulated in terms	of rel--$C^\infty$ maps. After presenting these facts, we will show that $\Mbar\subset\fM$ is open (in the case when $E=0$). The technique used for this also enables us to produce local (topological) \'etale charts for the moduli orbispace $\Mbar_{g,m}(X,J)_\beta$ of stable $J$-holomorphic maps of type $(g,m)$ in class $\beta$.
	\subsection{Families of Holomorphic Curves are Rel--$C^\infty$}
	Let $(\pi:\cC\to\cS,p_1,\ldots,p_m)$ be a proper flat analytic family of prestable curves of type $(g,m)$ as in \textsection\ref{mod-functor}. Let $\cC^\circ\subset\cC$ be the open subset defined as the complement of the union of the nodal loci of all the fibres of $\pi$. Since $\cC^\circ$ and $\cS$ are complex manifolds and $\pi:\cC^\circ\to\cS$ is a holomorphic submersion, it can be viewed canonically as a rel--$C^\infty$ manifold (with fibres of real dimension $2$).
	\begin{definition}[Rel--$C^\infty$ Structure on a Family of Curves]
		Suppose $\varphi:T\to\cS$ is a map of topological spaces. Let $\pi_T:\cC_T\to T$ be the base change of $\pi$ along $\varphi$. Let $\cC^\circ_T\subset\cC_T$ be the pullback of $\cC^\circ\subset\cC$. Then, $\pi_T:\cC^\circ\to T$ carries a rel--$C^\infty$ structure, pulled back from $\pi:\cC^\circ\to\cS$. We call this the \textbf{natural rel--$C^\infty$ structure} on $\cC^\circ_T/T$.
	\end{definition}
	\begin{remark}
		It is possible to verify that this rel--$C^\infty$ structure is independent of the presentation of $\cC_T/T$ as the pullback of $\cC/\cS$. More precisely, if we have any other $\cC'/\cS'$ and maps $\cC_T/T\to\cC'/\cS'$ fitting into a Cartesian square of topological spaces (respecting the fibrewise analytic structures), then the rel--$C^\infty$		structure induced by $\cC'/\cS'$ is the same as the one induced by $\cC/\cS$. Moreover, $\cC^\circ_T/T$ is, 		\emph{locally on the source} $\cC^\circ_T$, isomorphic to the trivial family $\bD_T$ (preserving the fibrewise analytic structures).
	\end{remark}
	\begin{lemma}\label{elliptic-reg}
		Suppose $F:\cC_T\to X$ is a continuous map such that for each $t\in T$, the map $F|_{C_{\varphi(t)}}$ satisfies		the equation (\ref{inhom-pde}). Then, the restriction of $F$ to $\cC^\circ_T/T$ is rel--$C^\infty$.
	\end{lemma}
	\begin{proof}
		For any point in $\cC^\circ_T$, pick a local isomorphism of $\cC^\circ_T/T$ with an open subset of $\bD_T$ (preserving the fibrewise analytic structure). Thus, we can identify $F$ on this neighborhood with a continuous map $f:\bD\times T'\to X$ (where $T'\subset T$ is open). Viewing this as a family of maps $(f_t)_{t\in T'}$ and combining the Gromov-Schwarz Lemma (see \cite[p316, 1.3.A]{Gromov} or \cite[p223, Corollary 4.1.4]{Muller-GS}) with elliptic bootstrapping as in the proof of \cite[Proposition B.11.5]{Pardon-VFC}, we see that on any compact subset of $T'$, the family $t\mapsto f_t$ is bounded in $C^\infty_\text{loc}(\bD,X)$. Using these $C^\infty_\text{loc}$ bounds on $(f_t)_{t\in T'}$ and the continuity of $F$, an Arzel\`a-Ascoli type argument (see Lemma \ref{Arz-Asc} below) shows that $F$ is rel--$C^\infty$ as desired.
	\end{proof}
	\noindent The following lemma essentially says that ``functions which are bounded in $C^{k+1}$ and close in $C^0$ are also close in $C^k$". It can be viewed as an effective version of an Arzel\`a-Ascoli type argument.
	\begin{lemma}\label{Arz-Asc}
		Let $k\ge 0$, $m\ge 1$ be integers, and $M,\varepsilon>0$ be positive constants. Then, there exists a positive function $\gamma=\gamma_{k,m}(M,\varepsilon)>0$ such that
		\begin{align}
			\lim_{\varepsilon\to 0^+} \gamma_{k,m}(M,\varepsilon) = 0
		\end{align}
		and with the following significance. If $f,g:[-1,1]^m\to\bR$ are two $C^{k+1}$ functions		defined on the $m$-dimensional cube, satisfying $\|f\|_{C^{k+1}}\le M$, $\|g\|_{C^{k+1}}\le M$ and		$\|f-g\|_{C^0}\le\varepsilon$, then we have
		\begin{align}
			\|f-g\|_{C^k}\le\gamma.
		\end{align}
	\end{lemma}
	\begin{proof}
		Let's fix $m\ge 1$ and $M$. We will prove the result by induction on $k$. For $k=0$, we can simply take $\gamma_{0,m}(M,\varepsilon)=\varepsilon$. Let us now assume that the result is true for $k=r$ and prove it for $k=r+1$ (where $r\ge 0$ is some integer). We will first estimate $\|\partial_if-\partial_ig\|_{C^0}$ for each $1\le i\le m$. Let $x\in[-1,1]^m$. If $0\ne\lambda\in\bR$		is such that $x + \lambda e_i\in[-1,1]^m$, then the mean value theorem gives
		\begin{align}
			\left|\frac{f(x+\lambda e_i)-f(x)}{\lambda} - \partial_if(x)\right|\le \sup_{|t|\le|\lambda|}
			|\partial_i f(x+te_i) - \partial_if(x)|\le M|\lambda|
		\end{align}
		which when combined with the analogous inequality for $g$ and the $C^0$ bound on $f-g$ gives
		\begin{align}
			|\partial_if(x) - \partial_ig(x)|\le 2M|\lambda| + 2|\lambda|^{-1}\varepsilon.
		\end{align}
		If $\varepsilon\ge M$, then set $\gamma_{r+1,m}(M,\varepsilon) = 2\varepsilon$. Indeed, in this case we have $\|f-g\|_{C^{r+1}}\le 2M\le 2\varepsilon$. On the other hand, if $\varepsilon<M$ then we can find $\lambda\in\bR$ with $|\lambda|^2 = \varepsilon/M$ such that $x+\lambda e_i\in[-1,1]^m$, which yields
		\begin{align}
			\|\partial_if-\partial_ig\|_{C^0}\le 4\sqrt{M\varepsilon}.
		\end{align}
		Now, by induction, we have $\|f-g\|_{C^r}\le \gamma_{r,m}(M,\varepsilon)$ and
		\begin{align}
			\|\partial_if-\partial_ig\|_{C^r}\le\gamma_{r,M}(M,4\sqrt{M\varepsilon})
		\end{align}
		for $i=1,\ldots,m$. Thus, since we have the definition
		\begin{align}
			\|f-g\|_{C^{r+1}}\le \|f-g\|_{C^r} +\sum_{i=1}^m\|\partial_if-\partial_ig\|_{C^r}
		\end{align}
		we can take
		\begin{align}
			\gamma_{r+1,m}(M,\varepsilon) := \gamma_{r,m}(M,\varepsilon) + m\cdot\gamma_{r,m}
			(M,4\sqrt{M\varepsilon})
		\end{align}
		which, by definition, satisfies $\gamma_{r+1,m}(M,\varepsilon)\to 0$ as $\varepsilon\to0^+$. This completes the induction.
	\end{proof}
	\subsection{Completing the Proof of Theorem \ref{stable-rep}}
	Theorem \ref{stable-rep} is an immediate consequence of the Lemma \ref{stable-is-open} below. (We continue to use the notation from Lemma \ref{elliptic-reg} in its statement and proof.)
	\begin{lemma}[Stability is Open]\label{stable-is-open}
		Suppose $F:\cC_T\to X$ is a continuous map such that for each $t\in T$, the map $F|_{C_{\varphi(t)}}$ is $J$-holomorphic. Then, the set of stable points, i.e., $t\in T$ such that 
		\begin{align}\label{finite-auto}
			\#\Aut(C_{\varphi(t)},p_1(\varphi(t)),\ldots,p_m(\varphi(t)),F|_{C_\varphi(t)})<\infty
		\end{align} 
		is open in $T$.
	\end{lemma}
	\begin{proof}
		Let $t\in T$ be such that (\ref{finite-auto}) is satisfied. Set $s:=\varphi(t)$ and $F_t:=F|_{C_{\varphi(t)}}$. We claim that on (the normalization of) any unstable irreducible component of $(C_s,p_1(s),\ldots,p_m(s))$,		the map $F_t$ has full rank at some point. Indeed, if this were not the case, then by $J$-holomorphicity $F_t$ would be constant on this component, contradicting (\ref{finite-auto}). Thus, it is possible to choose a minimal collection of distinct smooth points $q_1,\ldots,q_r\in C_s\setminus\{p_1(s),\ldots,p_m(s)\}$	such that $(C_s,p_1(s),\ldots,p_m(s),q_1,\ldots,q_r)$ is stable and $F_t$ is an immersion at each $q_i$ ($1\le i\le r$).\\\\
		Choose smooth functions $h_1,\ldots,h_r:X\to\bC$ such that $h_i(F_t(q_i))=0$ and $d(h_i\circ F_t)_{q_i}$ is an isomorphism for each $1\le i\le r$. Since $F$ is rel--$C^1$ on $\cC^\circ_T/T$ (by Lemma \ref{elliptic-reg}), it follows using the implicit function theorem with continuous dependence on parameters (a corollary of Lemma \ref{rel-ift}) -- applied to the functions $h_i\circ F:\cC^\circ_T\to\bC$ for $1\le i\le r$ -- that we have continuous local sections $\tilde q_i:T\to\cC^\circ_T$ for $1\le i\le r$ satisfying
		\begin{itemize}
			\item $\tilde q_i(t) = q_i$.
			\item For all $t'$ near $t$, we have $(h_i\circ F_t')(\tilde q_i(t')) = 0$ and $d(h_i\circ F_t')_{\tilde q_i(t')}$	is an isomorphism.
			\item For all $t'$ near $t$, the points $p_1(\varphi(t')),\ldots,p_m(\varphi(t')), \tilde q_1(t'),\ldots,\tilde q_r(t')$ are all distinct.
			\item For all $t'$ near $t$, the curves $(C_{\varphi(t')},p_1(\varphi(t')),\ldots,p_m(\varphi(t')), \tilde q_1(t'),\ldots,\tilde q_r(t'))$ are stable of type $(g,m+r)$. This is because stable curves are open in the stack of all prestable curves, see \cite[Chapter X, Lemma 3.4]{ACG-moduli}.
		\end{itemize} 
		Now, it's immediate that for each $t'$ near $t$, the map $(C_{\varphi(t')},p_1(\varphi(t')),\ldots,p_m(\varphi(t'),	F|_{C_{\varphi(t')}})$ is stable: since the presence of a marked point $\tilde q_i(t')$ on any irreducible component of $C_{\varphi(t')}$ indicates that $F$ is non-constant on this component and, in fact, immersed at $\tilde q_i(t')$. This shows that all $t'$ near $t$ satisfy (\ref{finite-auto}).
	\end{proof}
	\subsection{\'Etale Charts for the Moduli Stack of Stable Maps}\label{et-chart}
	The technique used in the proof of Lemma \ref{stable-is-open} can also be used to produce local \'etale charts for $\Mbar_{g,m}(X,J)_\beta$, the moduli orbispace of stable	$m$-pointed $J$-holomorphic curves in $X$ of arithmetic genus $g$ lying in the homology class $\beta$. (This orbispace may be highly singular, and thus the domains for its \'etale charts also will be. Over the locus of regular stable maps, these may be viewed as orbifold charts.) \\\\
	Let $(C,x_1,\ldots,x_m,u)$ be a point of $\Mbar_{g,m}(X,J)_\beta$.  Let $y_1,\ldots,y_r$ be a minimal	collection of distinct smooth points on $C$ such that $(C,x_1,\ldots,x_m,y_1,\ldots,y_r)$ is stable and $u$ is immersed at each $x_i$ for $1\le i\le m$. Minimality implies	that $r$ is the $\bC$-dimension of the Lie algebra $\fa\fu\ft(C,x_1,\ldots,x_m)$ of infinitesimal automorphisms and $\fa\fu\ft(C,x_1,\ldots,x_m,y_1,\ldots,y_r) = 0$.
	\\\\
	Let $\pi:\cC_V\to V$ be a universal deformation of $(C,x_1,\ldots,x_m,y_1,\ldots,y_r)$, i.e., $\pi:\cC_\cV\to\cV$ is a proper flat analytic family of stable curves (with marked points given by sections $\sigma_1,\ldots,\sigma_m,\tau_1,\ldots,\tau_r$ of $\pi$) equipped with a point $p\in\cV$ and a Cartesian diagram
	\begin{center}
		\begin{tikzcd}
			C \arrow[d] \arrow[r,"\varphi"] & \cC_\cV \arrow[d, "\pi"] \\
			* \arrow[r, "p"] \arrow[u, bend left, "{x_i, y_j}"] & \cV \arrow[u, bend left, "{\sigma_i,\tau_j}"]
		\end{tikzcd}
	\end{center}
	such that the following property holds. Given any other continuous family $(\pi_T:\cC_T\to T,\{\sigma'_i\}_{i=1}^m,\{\tau'_j\}_{j=1}^r)$ of stable curves (which is, locally on the topological space $T$, the pullback of an proper analytic family of stable curves), along with a point $q\in T$ and a fibre square
	\begin{center}
		\begin{tikzcd}
			C \arrow[d] \arrow[r,"\psi"] & \cC_T \arrow[d, "\pi"] \\
			* \arrow[r, "q"] \arrow[u, bend left, "{x_i, y_j}"] & T \arrow[u, bend left, "{\sigma_i',\tau_j'}"]
		\end{tikzcd}
	\end{center}
	there exists an open neighborhood $T'\subset T$ of $q$ and a fibre square
	\begin{center}
		\begin{tikzcd}
			\cC_T \arrow[d, "\pi_T"] \arrow[rr,"\tilde f"] & & \cC_\cV \arrow[d, "\pi"] \\
			T \arrow[rr, "f"] \arrow[u, bend left, "{\sigma_i', \tau_j'}"] & & \cV \arrow[u, bend left, "{\sigma_i,\tau_j}"]
		\end{tikzcd}
	\end{center}
	which is compatible with the two fibre squares above, i.e., $f(q)=p$ and $\tilde f\circ\psi = \varphi$. Moreover, the germs of $f$ and $\tilde f$ at $q$ are unique. Next, let $h_1,\ldots,h_r:X\to\bC$ denote smooth functions such that for each $1\le j\le r$, we have $h_j(u(y_j)) = 0$ and $d(h_j\circ u)_{y_j}$ is an isomorphism. (Thus, $H_j:=h_j^{-1}(0)$ is a smooth codimension 2 submanifold of $X$ near $u(y_j)$ and meets $u$ transversely at $y_j$ for $1\le j\le m$.)\\\\
	Let $\Mbar:=\Mbar_J(\pi,X)_\beta^\text{top}$ be the $\cV$-space we get by the extension (as explained in Remark \ref{non-reg}) of Lemma \ref{top-moduli} applied to the family $\pi:\cC_\cV\to\cV$. We have the corresponding universal family of stable maps
	\begin{center}
		\begin{tikzcd}
			\cC_\cV \arrow[d, "\pi"] & \arrow[l] \cC_{\Mbar} \arrow[d, "{\pi_{\Mbar}}"] \arrow[r,"\text{ev}_\pi"] & X\\
			\cV & \arrow[l] \Mbar &
		\end{tikzcd}
	\end{center}
	Now, consider the continuous function $h:\Mbar\to\bC^r$ defined by
	\begin{align}
		(v\in\cV, g:\pi^{-1}(v)\to X)\mapsto \left((h_1\circ g\circ \sigma_1)(v),\ldots,(h_r\circ g\circ \sigma_r)(v)\right).
	\end{align}
	Let $\Mbar':=h^{-1}(0)\subset\Mbar$ be the subset where $h$ is zero. Let $\Mbar_\text{reg}\subset\Mbar$ be the open subset consisting of the regular stable maps and by Theorem \ref{stable-rep}, $\Mbar_\text{reg}/\cV$ naturally carries the structure of a rel--$C^\infty$ manifold.
	\begin{definition}[Regular Locus of $\Mbar'_\text{reg}$]
		Note that $h$ is a rel--$C^\infty$ function on the rel--$C^\infty$ manifold $\Mbar_\text{reg}/\cV$. We define
		$\Mbar'_\text{reg}\subset\Mbar'$ to be the open subset consisting of pairs $(v,g)$ such that $g$ is a regular stable map $\pi^{-1}(v)\to X$ and the linearization $T_h|_{(v,g)}:T_{\cM_\text{reg}/\cV}|_{(v,g)}\to\bC^r$ is surjective. It's immediate from the implicit function theorem (with continuous dependence on parameters) that		$\Mbar'_\text{reg}$ is a locally closed rel--$C^\infty$ submanifold of $\Mbar_\text{reg}$ over $\cV$.	
	\end{definition}
	\noindent On $\Mbar'/\cV$, there is a corresponding family
	\begin{center}
		\begin{tikzcd}
			\cC_\cV \arrow[d, "\pi"] & \arrow[l] \cC_{\Mbar'} \arrow[d, "{\pi_{\Mbar'}}"] \arrow[r,"\text{ev}'_\pi"] & X\\
			\cV & \arrow[l] \Mbar' &
		\end{tikzcd}
	\end{center}
	of stable maps, restricted from the family $\text{ev}_\pi$ defined over $\Mbar$.
	\begin{lemma}
		If $(p,u)\in\normalfont\Mbar_\text{reg}$, then we have $(p,u)\in\normalfont\Mbar'_\text{reg}$.
	\end{lemma}
	\begin{proof}
		Note that $h(p,u)=0$. We therefore only need to prove the surjectivity of $T_h|_{(p,u)}$. Let's note that	$T_{\Mbar/\cV}|_{(p,u)}$ can be canonically identified with the kernel of the surjective linear operator
		\begin{align}
			 \Omega^0(C,u^*T_X)\xrightarrow{D(\delbar_J)_u} \Omega^{0,1}_J(\tilde C, u^*T_X)
		\end{align}
		and $T_h|_{(p,u)}$ can be identified with the map 
		\begin{align}
			\cL:\ker D(\delbar_J)_u&\to\bC^r\\
			\xi&\mapsto (dh_1(\xi(y_1)),\ldots,dh_r(\xi(y_r))).
		\end{align}
		Now, observe that we have a commutative diagram
		\begin{center}
			\begin{tikzcd}
				H^0(\tilde C, T_{\tilde C}(-D)) \arrow[r, "du"] \arrow[d, "\text{ev}"] & 
				\ker D(\delbar_J)_u \arrow[d, "\cL"]\\
				\displaystyle\bigoplus_{j=1}^r T_{\tilde C,y_j} \arrow[r, "\bigoplus_{j=1}^rd(h_j\circ u)_{y_j}"] & \bC^r
			\end{tikzcd}
		\end{center}
		where $D$ is the divisor on the normalization $\tilde C\to C$ consisting of the preimages of the nodes of $C$		and the marked points $x_1,\ldots,x_m$ and $\text{ev}$ represents the evaluation map at the points $y_1,\ldots,y_r$. Note that $H^0(\tilde C,T_{\tilde C}(-D)) = \fa\fu\ft(C,x_1,\ldots,x_m)$ and thus, the map		$\text{ev}$ is injective (and thus, an isomorphism by dimension considerations). The bottom horizontal arrow is an isomorphism by definition, and thus, it follows that $\cL$ is surjective.
	\end{proof}
	\noindent The family of stable maps $\text{ev}'_\pi$ defined over $\Mbar'$ has the following universal property. 
	\begin{proposition}[\'Etale Chart]\label{et}
		Let $\pi_T:\cC_T\to T$ be a family of $m$-pointed prestable curves of genus $g$ (locally given as the pull back of a proper flat analytic family) with marked points given by sections $\sigma'_1,\ldots,\sigma'_m$. Let $F:\cC_T\to X$ be a continuous map which is $J$-holomorphic on each fibre of $T$ and $q\in T$ be a point. Assume we have a commutative diagram
		\begin{center}
		\begin{tikzcd}
			& & X\\
			C \arrow[d] \arrow[urr, bend left, "u"] \arrow[r,"\psi"] & \cC_T \arrow[d, "\pi"] \arrow[ur, "F"]&\\
			* \arrow[r, "q"] \arrow[u, bend left, "{x_i}"] & T \arrow[u, bend left, "{\sigma_i'}"] &
		\end{tikzcd}
	\end{center}
	where the square is Cartesian. Then, there exists an open neighborhood of $q$ (still denoted $T$) and continuous maps $f:T\to \Mbar'$ and $\tilde f:\cC_T\to\cC_{\Mbar'}$ (with the latter fibrewise holomorphic) such that the diagram
	\begin{center}
		\begin{tikzcd}
             & X  &  &  \\
             &     &  &  \\
             & C \arrow[uu, "u"'] \arrow[ld, "\psi"'] \arrow[dd] \arrow[rrd, "{(\varphi,u)}"] &  &  \\
		\cC_T \arrow[ruuu, "F", bend left] \arrow[rrr, "\tilde f"' near end, crossing over] \arrow[dd, "\pi_T"'] &   
		&  & \cC_{\Mbar'} \arrow[dd, "\pi_{\Mbar'}"'] \arrow[lluuu, "\normalfont\text{ev}'_\pi"', bend right] \\
             & * \arrow[ld, "q"] \arrow[rrd, "{(p,u)}"] \arrow[uu, "x_i"' near start, bend right, crossing over]   &  &  \\
		T \arrow[rrr, "f"' near end] \arrow[uu, "\sigma_i'", bend right, ]  & &  & \Mbar' \arrow[uu, "\sigma_i", bend right]                              
		\end{tikzcd}
	\end{center}
	commutes (with all squares being Cartesian). Moreover, the germs of $f$ and $\tilde f$ near $q$ are unique.
	\end{proposition}
	\begin{proof}
		Let us first prove the existence of $f$ and $\tilde f$. Exactly as in the proof of Lemma \ref{stable-is-open},		there are unique (germs of) continuous sections $\tau'_j$ defined near $q\in T$ such that we have $\tau'_j(q) = \psi(y_j)$ and $h_j\circ F\circ\tau'_j\equiv 0$ (for each $1\le j\le r$). Now, by the universal property of the universal family $\pi:\cC_\cV\to\cV$, there exists unique (germs near $q$ of) maps $g,\tilde g$ such that the diagram
		
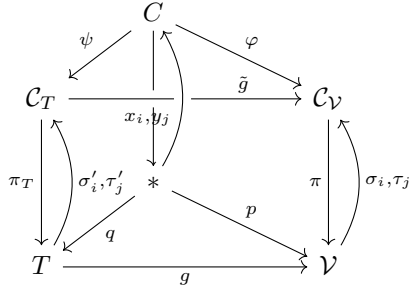
\begin{center}
		\begin{tikzcd}
             & C \arrow[ld, "\psi"'] \arrow[dd] \arrow[rrd, "\varphi"] &  &  \\
		\cC_T \arrow[rrr, "\tilde g" near end, crossing over] \arrow[dd, "\pi_T"'] &                                                                           
		&  & \cC_\cV \arrow[dd, "\pi"'] \\
             & * \arrow[ld, "q"] \arrow[rrd, "p"] \arrow[uu, "{x_i,y_j}" near start, bend right, crossing over]   &  &  \\
		T \arrow[rrr, "g"' below] \arrow[uu, "{\sigma_i',\tau'_j}" right, bend right, ]  & &  & 
		\cV \arrow[uu, "{\sigma_i,\tau_j}" right, bend right]                                 
		\end{tikzcd}
		\end{center}
		is commutative and the squares are Cartesian. By the universal property of $\Mbar = \Mbar_J(\pi,X)^\text{top}$,	we get unique maps $f$ and $\tilde f$ so that the diagram in the statement of the proposition commutes (with every instance of $\Mbar'$ replaced by $\Mbar$). However, by our choice of the sections $\tau_j'$, we see that $h\circ f = 0$ and thus, $f$ maps into $\Mbar'$. This proves the existence of $f$ and $\tilde f$ with the desired properties.\\\\
		Conversely, given maps $f'$ and $\tilde f'$ making the diagram in the statement of the proposition commute, we		will show that they are the same as $f$ and $\tilde f$ near $q$. The maps $f',\tilde f'$ determine maps $g',\tilde g'$ to $\cV,\cC_\cV$ respectively. Pulling back the sections $\tau_j$ of $\pi$ (along $g',\tilde g'$), we get sections $\tau_j''$ of $\pi_T$ satisfying $\tau''_j(q)=\psi(y_j)$. The crucial observation now is that since $f$ maps $T$ into $\Mbar' = h^{-1}(0)\subset\Mbar$, it follows that $h_j\circ F\circ \tau_j''\equiv 0$, which implies $\tau_j''\equiv \tau_j'$. This implies that, as germs near $q$, we have $(g,\tilde g) \equiv (g',\tilde g')$ and thus, $(f,\tilde f) \equiv (f',\tilde f')$.
	\end{proof}
	\begin{remark}
	For readers familiar with the language of topological stacks, Proposition \ref{et} just says that the map of stacks $\Mbar'\to\Mbar_{g,m}(X,J)_\beta$ -- defined by the family $\text{ev}'_\pi$ of stable maps over $\Mbar'$ -- is \'etale, i.e. a local homeomorphism, at the point $(p,u)$.
	\end{remark}
	\appendix
	\section{Weak Rel--$C^r$ Morphisms}\label{wk-appx}
		\subsection{Definitions}
		\label{wkrel}
		We define an auxiliary notion of rel-$C^r$ morphisms for relative Banach manifolds which generalizes the notion already defined in finite (relative) dimensions.
		\begin{definition}
			Let $S$ be a topological space and $V,W$ be $C^\infty$ Banach manifolds. If $U\subset V\times S$ is open, then a map $F:U\subset V\times S\to W\times S$ is called 
			\begin{enumerate}[(1)]
				\item a \textbf{weak rel--$C^0$ $S$-morphism} if it is continuous on $U$ and of the form
				\begin{align}
					(v,s)\mapsto (f(v,s),s).
				\end{align}
				\item a \textbf{weak rel--$C^r$ $S$-morphism} for an integer $r\ge 1$ if the following hold.
				\begin{itemize}
					\item It is a weak rel--$C^0$ $S$-morphism.
					\item For each $s\in S$, the map $v\mapsto f(v,s)$ is of class $C^1$.
					\item The map $dF:(TV\times S)|_U\to TW\times S$ on the tangent bundles (induced				by differentiating $F$ in the $V$-directions) is a weak rel--$C^{r-1}$ $S$-morphism.
				\end{itemize}
				\item a \textbf{weak rel--$C^\infty$ $S$-morphism} if it is a weak rel--$C^r$ $S$-morphism for all
				integers $r\ge 0$.
			\end{enumerate}
		\end{definition}
		\begin{remark}\label{wk-properties} 
		We note a few assertions which are evident.
		\begin{enumerate}[(1)]
			\item ($C^r$ is weak rel--$C^r$.) If $S$ is also a $C^\infty$ Banach manifold and the map $F$ is $C^r$ and a weak rel--$C^0$ $S$-morphism, then it is also a weak rel--$C^r$ $S$-morphism.
			\item (Rel--$C^r$ is weak rel--$C^r$.) If the Banach manifolds $V,W$ are finite dimensional, then weak rel--$C^r$ $S$-morphisms are the same as rel--$C^r$ $S$-morphisms $U\subset V\times S\to W\times S$.
			\item (Chain Rule.) The composition of two weak rel--$C^r$ $S$-morphisms is again a	weak rel--$C^r$ $S$-morphism.
			\item (Sheaf Property.)	A map $F:U\subset V\times S\to W\times S$ is a weak rel--$C^r$ $S$-morphism if there is an open cover $\{U_\alpha\}_{\alpha\in I}$ of $U$ such that $F|_{U_\alpha}$ is a weak rel--$C^r$ $S$-morphism for all $\alpha\in I$. 
		\end{enumerate}
		\end{remark}
		\noindent We also note the following useful fact.
		\begin{lemma}\label{wk-linear}
			Suppose $L:V\times S\to W\times S$, given by $L(v,s) = (L_s(v),s)$ is a continuous map, with $L_s:V\to W$ being a (bounded) linear map of Banach spaces for every $s\in S$. Then, $L$ is a weak rel--$C^\infty$ $S$-morphism.
		\end{lemma}
		\begin{proof}
			We proceed by induction on $r\ge 0$ to show that $L$ is a weak rel--$C^r$ $S$-morphism. The base case $r=0$ is obvious, so let $r\ge 1$. Let us note that for each $s\in S$, since $L_s$ is linear, it is $C^1$ as a map	from $V\to W$ and the map $dL$ is given by
			\begin{align}
				dL:V\times V\times S\to W\times W\times S\\
				(v,\delta v,s)\mapsto (L_s(v), L_s(\delta v), s)
			\end{align}
			where we have used the identification $TV = V\times V$ and $TW = W\times W$. Clearly, $dL$ is continuous, and thus, by the induction hypothesis, it must be a weak rel--$C^{r-1}$ $S$-morphism. Thus, $L$ is a weak rel--$C^r$ $S$-morphism. This completes the induction.
		\end{proof}
		\subsection{Weak Relative Smoothness of Evaluation Maps}
		We follow the notation of \textsection\ref{sc-ift} (specifically, \textsection\ref{ift-chart}) in this appendix. Our aim here is to establish the weak relative smoothness of evaluation maps.
		\begin{lemma}\label{wkrel-ev}
			The map $e_2$ from (\ref{e2}) is a weak rel--$C^\infty$ $\cC$-morphism.
		\end{lemma}
		\begin{proof}
			By dividing $\cC$ into the necks/ends and a neighborhood of the complement of the necks/ends		(which, on each fibre, may be identified with a neighborhood of $C'$) and using Remark \ref{wk-properties}(4),	the result will follow from Lemmas \ref{ev-int} and \ref{ev-node}.
		\end{proof}
		\begin{lemma}\label{ev-int}
			Let $M$ be a compact manifold and let $Y$ be any manifold of bounded geometry. Define $H^k(M,Y)$ to		be the space of continuous maps $M\to Y$ which are of class $W^{k,2}$ where $2k>\dim M$. Then, the map
			\begin{align}
				e_{k,M,Y}:H^k(M,Y)\times M\to Y\times M\\
				(f,z)\mapsto (f(z),z)
			\end{align}
			is a weak rel--$C^\infty$ $M$-morphism.
		\end{lemma}
		\begin{proof}
			Continuity of $e_{k,M,Y}$ is a direct consequence of the Sobolev embedding $W^{k,2}\subset C^0$ on $M$.		It's also immediate that the evaluation map $H^k(M,Y)\to Y$ at each $m\in M$ is $C^1$. Moreover,	we see that $d(e_{k,M,Y}) = e_{k,M,TY}$. Thus, by using induction, we may conclude that $e_{k,M,Y}$ is a weak rel--$C^\infty$ $M$-morphism.
		\end{proof}
		\noindent In preparation for the next Lemma, define the weighted Sobolev space $H^{k,\delta}_\pm$ as follows.		Elements of $H^{k,\delta}_\pm$ consist of pairs $(\xi,\xi')$ where
		\begin{itemize}
			\item $\xi,\xi':(0,\infty)\times S^1\to\bR$ are continuous functions. On the domain of $\xi$, we denote	the coordinates by $(s,t)$, while on the domain of $\xi'$, we denote the coordinates by $(s',t')$.
			\item There exists $\xi_\infty\in\bR$ such that $\xi-\xi_\infty$ and $\xi'-\xi_\infty$ are both of class $W^{k,2,\delta}$ on $(0,\infty)\times S^1$.
			\item the norm squared of $(\xi,\xi')$ is given by
			\begin{align}
				|\xi_\infty|^2 + \|\xi-\xi_\infty\|^2_{W^{k,2,\delta}} + \|\xi'-\xi_\infty\|^2_{W^{k,2,\delta}}.
			\end{align}
		\end{itemize}
		Now, let $\cU$ be the set of pairs $(z,z')\in\bC^2$ such that $|z|<1$, $|z'|<1$ and $|zz'|<1/4$. Given $(z,z')\in\cU$ and $(\xi,\xi')\in H^{k,\delta}_\pm$, we will define an element $e(z,z',\xi,\xi')\in\bR$ as follows. Let $\alpha:=zz'$. We consider two cases.
		\begin{enumerate}[(1)]
			\item If $\alpha = e^{-(R+i\theta)}\ne 0$, then define glue the cylinder $(0,R)\times S^1$ with coordinates $(s,t)$ to the cylinder $(0,R)\times S^1$ with coordinates $(s',t')$ using the equations $s+s' = R$ and $t+t' = \theta$. We then define $e(z,z',\xi,\xi')$ to be
			\begin{align}\label{e1}
				\beta_R(s)\xi(s,t) + \beta_R(s')\xi'(s',t')
			\end{align}
			evaluated at the points $(s,t)$ and $(s',t')$ satisfying $z = e^{-(s+it)}$ and $z'=e^{-(s'+it')}$.
			\item If $\alpha = 0$, we proceed as follows. If $z = 0$ and $z'\ne 0$, we define $e(z,z',\xi,\xi')$ to be
			\begin{align}
				\xi'(s',t')
			\end{align}
			evaluated at the point $(s',t')$ satisfying $z' = e^{-(s'+it')}$. We define $e(z,z',\xi,\xi')$ analogously if $z\ne 0$ and $z'=0$. Finally, if $z=z'=0$, we define $e(z,z',\xi,\xi'):=\xi_\infty$.
		\end{enumerate}
		\begin{lemma}\label{ev-node}
			The map 
			\begin{align}
				e_1:\cU\times H^{k,\delta}_\pm&\to\cU\times\bR\\
				(z,z',\xi,\xi')&\mapsto(z,z',e(z,z',\xi,\xi'))
			\end{align} 
			is a weak rel--$C^\infty$ $\cU$-morphism.
		\end{lemma}
		\begin{proof}
			Note that the map $e_1$ is of the form considered in Lemma \ref{wk-linear}. Thus, it will be enough to		show that $e$ is continuous. Writing $\xi = \zeta + \xi_\infty$ and $\xi' = \zeta' + \xi_\infty$, we get	the exponential decay estimate
			\begin{align}
				|\zeta(s,t)|&\le Ce^{-\delta s}\|\zeta\|_{k,\delta}\\ 
				|\zeta'(s',t')|&\le Ce^{-\delta s'}\|\zeta'\|_{k,\delta}
			\end{align} 
			for some suitable constant $C>0$ (independent of $\xi,\xi'$) coming from the Sobolev embedding $W^{k,2}\subset C^0$. Let's note that we can slightly change the formula (\ref{e1}) for $e(z,z',\xi,\xi')$, so that it can be interpreted suitably even when $\alpha:=zz'=0$, as follows
			\begin{align}
				e(z,z',\xi,\xi') = \xi_\infty + \beta_R(s)\zeta(s,t) + \beta_R(s')\zeta'(s',t').
			\end{align}
			To interpret this formula when $\alpha:=zz'=0$, we agree to set $R=\infty$ and $\beta_\infty\equiv 1$ in	this case. When $z = 0$, we set $\zeta(s,t):=0$ and when $z'=0$, we set $\zeta'(s,t):=0$ (this can be interpreted as setting $s=\infty$ and leaving $t$ undefined and, respectively, setting $s'=0$ and leaving $t'$ undefined). \\\\
			Now, let's consider a sequence $(z_n,z_n',\xi,\xi')\to(z,z',\xi,\xi')\in\cU\times H^{k,\delta}_\pm$. We immediately note that $\xi_{n,\infty}\to\xi_\infty$ and thus, we need to only show the convergence $\beta_{R_n}(s_n)\zeta_n(s_n,t_n)\to \beta_R(s)\zeta(s,t)$ for the second term (and the analogous statement for the third term). We deal only with the second term here (the argument for the third 
			term is symmetric). If $\alpha := zz'\ne 0$, then we obviously have $\beta_{R_n}(s_n)\to\beta_R(s)$ and
			\begin{align}\label{ptwise}
				\|\zeta_n(s_n,t_n)-\zeta(s,t)\|\le Ce^{-\delta s_n}\|\zeta_n-\zeta\|_{k,\delta}+
				\|\zeta(s_n,t_n)-\zeta(s,t)\|\to 0
			\end{align}
			which completes the proof in this case. Now, suppose $\alpha = zz'=0$ in which case we have $\beta_R(s)=1$.	Now, if $z\ne 0$, then we have $s_n\to s<\infty$ and $t_n\to t\in S^1$ which gives $\beta_{R_n}(s_n)=1$ for large $n$ (since $R_n\to 0$) and $\|\zeta_n(s_n,t_n)-\zeta(s,t)\|\to 0$ as in (\ref{ptwise}). Finally, let's consider the case when $z=0$. In this case, $\beta_R(s)\zeta(s,t)=0$, $\|\zeta_n\|_{k,\delta}\to\|\zeta\|_{k,\delta}$ and we have $s_n\to\infty$ and thus,
			\begin{align}
				\|\beta_{R_n}(s_n)\xi_n(s_n,t_n)\|\le\|\xi_n(s_n,t_n)\|\le Ce^{-\delta s_n}\|\xi_n\|_{k,\delta}\to 0
			\end{align}
			which completes the proof of continuity in this case as well.
		\end{proof}
	\addcontentsline{toc}{section}{References}
	\bibliographystyle{amsalpha}
	\bibliography{rel-smooth-gluing}

\providecommand{\bysame}{\leavevmode\hbox to3em{\hrulefill}\thinspace}
\providecommand{\MR}{\relax\ifhmode\unskip\space\fi MR }
\providecommand{\MRhref}[2]{%
  \href{http://www.ams.org/mathscinet-getitem?mr=#1}{#2}
}
\providecommand{\href}[2]{#2}
\begin{thebibliography}{FFGW16}

\bibitem[Abo12]{Abou}
M.~Abouzaid, \emph{{Framed bordism and Lagrangian embeddings of exotic
  spheres.}}, {Ann. Math. (2)} \textbf{175} (2012), no.~1, 71--185 (English).

\bibitem[ACG11]{ACG-moduli}
E.~Arbarello, M.~Cornalba, and P.~A. Griffiths, \emph{Geometry of algebraic
  curves. {V}olume {II}}, Grundlehren der Mathematischen Wissenschaften
  [Fundamental Principles of Mathematical Sciences], vol. 268, Springer,
  Heidelberg, 2011, With a contribution by Joseph Daniel Harris. \MR{2807457}

\bibitem[CJS95]{Floer-Homotopy}
R.~L. Cohen, J.~D.~S. Jones, and G.~B. Segal, \emph{{Floer's infinite
  dimensional Morse theory and homotopy theory.}}, {The Floer memorial volume},
  Basel: Birkh\"auser, 1995, pp.~297--325 (English).

\bibitem[ES16]{Ek-Sm}
T.~Ekholm and I.~Smith, \emph{{Exact Lagrangian immersions with a single double
  point.}}, {J. Am. Math. Soc.} \textbf{29} (2016), no.~1, 1--59 (English).

\bibitem[FFGW16]{FFGW-polyfold-survey}
O.~Fabert, J.~W. Fish, R.~Golovko, and K.~Wehrheim, \emph{Polyfolds: a first
  and second look}, EMS Surv. Math. Sci. \textbf{3} (2016), no.~2, 131--208.
  \MR{3576532}

\bibitem[FOOO15]{FOOO1}
K.~Fukaya, Y.-G. Oh, H.~Ohta, and K.~Ono, \emph{{Kuranishi structure,
  pseudo-holomorphic curve, and virtual fundamental chain: Part 1}}, arXiv
  e-prints (2015), arXiv:1503.07631.

\bibitem[FOOO16]{FOOO-gluing}
\bysame, \emph{{Exponential decay estimates and smoothness of the moduli space
  of pseudoholomorphic curves}}, arXiv e-prints (2016), arXiv:1603.07026.

\bibitem[FOOO17]{FOOO2}
\bysame, \emph{{Kuranishi structure, pseudo-holomorphic curve, and virtual
  fundamental chain: Part 2}}, arXiv e-prints (2017), arXiv:1704.01848.

\bibitem[Gro85]{Gromov}
M.~Gromov, \emph{Pseudo holomorphic curves in symplectic manifolds}, Invent.
  Math. \textbf{82} (1985), no.~2, 307--347. \MR{809718}

\bibitem[HWZ07]{HWZ-Splicing}
H.~Hofer, K.~Wysocki, and E.~Zehnder, \emph{A general {F}redholm theory {I}: A
  splicing-based differential geometry}, JEMS, Volume 9, Issue (2007), 876.

\bibitem[HWZ09a]{HWZ-ImpFuncThms}
\bysame, \emph{A general {F}redholm theory {II}: Implicit function theorems},
  Geometric and Functional Analysis \textbf{19} (2009), no.~1, 206--293.
  \MR{2507223}

\bibitem[HWZ09b]{HWZ-Fred}
\bysame, \emph{A general {F}redholm theory {III}: Fredholm functors and
  polyfolds}, Geom. Topol. \textbf{13} (2009), no.~4, 2279--2387.

\bibitem[HWZ10]{HWZ-NewModels}
\bysame, \emph{Sc-smoothness, retractions and new models for smooth spaces},
  Discrete and Continuous Dynamical Systems \textbf{28} (2010).

\bibitem[HWZ14]{HWZ-PFBasics}
\bysame, \emph{{Polyfold and Fredholm Theory I: Basic Theory in M-Polyfolds}},
  arXiv e-prints (2014), arXiv:1407.3185.

\bibitem[HWZ17]{HWZ-GWbook}
\bysame, \emph{Applications of polyfold theory {I}: {T}he polyfolds of
  {G}romov-{W}itten theory}, Mem. Amer. Math. Soc. \textbf{248} (2017),
  no.~1179, v+218. \MR{3683060}

\bibitem[Mul94]{Muller-GS}
M.-P. Muller, \emph{Gromov's {S}chwarz lemma as an estimate of the gradient for
  holomorphic curves}, Holomorphic curves in symplectic geometry, Progr. Math.,
  vol. 117, Birkh\"{a}user, Basel, 1994, pp.~217--231. \MR{1274931}

\bibitem[MW17]{MW}
D.~McDuff and K.~Wehrheim, \emph{Smooth {K}uranishi atlases with isotropy},
  Geom. Topol. \textbf{21} (2017), no.~5, 2725--2809.

\bibitem[Pal68]{Palais-FNGA}
R.~S. Palais, \emph{Foundations of global non-linear analysis}, W. A. Benjamin,
  Inc., New York-Amsterdam, 1968. \MR{0248880}

\bibitem[Par16]{Pardon-VFC}
J.~Pardon, \emph{An algebraic approach to virtual fundamental cycles on moduli
  spaces of pseudo-holomorphic curves}, Geom. Topol. \textbf{20} (2016), no.~2,
  779--1034. \MR{3493097}

\bibitem[Rud76]{Rudin}
W.~Rudin, \emph{Principles of mathematical analysis}, McGraw-Hill Science
  Engineering Math, 1976.

\bibitem[SPa20]{stacks-project}
The {S}tacks~{P}roject authors, \emph{The {S}tacks {P}roject},
  \url{https://stacks.math.columbia.edu}, 2020.

\end{thebibliography}
\end{document}